\documentclass{amsart}
\usepackage{amsmath,amssymb}
\usepackage{amsfonts}
\usepackage{amsthm}
\usepackage{mathrsfs}
\usepackage{comment}
\newcommand{\jap}[1]{\langle #1 \rangle}

\def\a{\alpha}
\def\b{\beta}
\def\c{\gamma}
\def\d{\delta}
\def\e{\varepsilon}
\def\f{\varphi}
\def\g{\psi}

\def\k{\kappa}
\def\l{\lambda}
\def\m{\mu}

\def\s{\sigma}

\def\x{\xi}
\def\y{\eta}

\renewcommand{\L}{\Lambda}

\def\re{\mathbb{R}}

\def\ze{\mathbb{Z}}

\def\T{\mathbb{T}}
\def\pa{\partial}

\renewcommand{\Re}{\text{{\rm Re}\;}}
\renewcommand{\Im}{\text{{\rm Im}\;}}

\newcommand{\supp}{\text{{\rm supp}\;}}

\newcommand{\Ker}{\text{{\rm Ker}\;}}
\newcommand{\Ran}{\text{{\rm Ran}\;}}

\newcommand{\sgn}{\mathrm{sgn}}

\newtheorem{thm}{Theorem}[section]
\newtheorem{lem}[thm]{Lemma}
\newtheorem{prop}[thm]{Proposition}
\newtheorem{cor}[thm]{Corollary}

\theoremstyle{definition}

\newtheorem{ass}{Assumption}

\newtheorem{example}{Example}

\theoremstyle{remark}
\newtheorem{rem}[thm]{Remark}

\title{Limiting absorption principle on $L^p$-spaces and scattering theory}

\author{Kouichi Taira}
\address{Graduate School of Mathematical Sciences, University of Tokyo, 3-8-1 Komaba, Meguroku, Tokyo, Japan 153-8914}
\email{taira@ms.u-tokyo.ac.jp}

\begin{document}

\maketitle

\begin{abstract}
In this paper, we study the mapping property form $L^p$ to $L^q$ of the resolvent of the Fourier multipliers and scattering theory of generalized Schr\"odinger operators. Though the first half of the subject is studied in \cite{C2}, we extend their result to away from the duality line and we also study the H\"older continuity of the resolvent.
\end{abstract}

\section{Introduction}

In this note, we study $L^p$-estimates for resolvents of the Fourier multipliers and the scattering theory of the discrete Schr\"odinger operator, the fractional Schr\"odinger operators and the Dirac operators.

One of the interest in the scattering theory of the Schr\"odinger operator is to prove the asymptotic completeness of the wave operators:
\begin{align*}
W_{\pm}=s-\lim_{t\to \pm \infty}e^{it(-\Delta+V)}e^{-it(-\Delta)},
\end{align*}
i.e. that $W_{\pm}$ are surjections onto the absolutely continuous subspace of $L^2(\re^d)$. Through the Kato's smooth perturbation theory, the asymptotic completeness of the wave operators is closely related to the limit absorption principle:
\begin{align}
\sup_{z\in I_{\pm}\setminus I}\||V|^{\frac{1}{2}}(-\Delta-z)^{-1}|V|^{\frac{1}{2}}\|_{B(L^2(\re^d))}<&\infty, \label{lim}\\
\sup_{z\in I_{\pm}\setminus I}\||V|^{\frac{1}{2}}(-\Delta+V-z)^{-1}|V|^{\frac{1}{2}}\|_{B(L^2(\re^d))}<&\infty, \label{lim2} 
\end{align}
where $I\subset (0,\infty)$ is an interval and $I_{\pm}=\{z\in \mathbb{C}\mid \pm \Im z\geq 0\}$ and $V$ is a real-valued function. A strong tool for proving $(\ref{lim})$ and $(\ref{lim2})$ is the Mourre theory \cite{M}, which gives sufficient conditions that $(\ref{lim})$ and $(\ref{lim2})$ hold.

On the other hands, Kenig, Ruiz and Sogge \cite{KRS} establish the $L^p$-type limiting absorption principle for the free Schr\"odinger operator:
\begin{align}\label{L^plim}
\|(-\Delta-z)^{-1}\|_{B(L^p(\re^d), L^{q}(\re^d))}\leq C_{p,q}|z|^{\frac{d}{2}(\frac{1}{p}-\frac{1}{q})-1},\,\, z\in \mathbb{C}\setminus [0,\infty),\,\, d\geq 3
\end{align}
where $C_{p,q}>0$ is independent of $z\in \mathbb{C}\setminus [0,\infty)$ and $(1/p,1/q)\in (0,1)\times (0,1)$ satisfies $2/(d+1)\leq 1/p-1/q\leq 2/d$, $(d+1)/2d<1/p$ and $1/q<(d-1)/(2d)$. $(\ref{L^plim})$ is also proved by Kato and Yajima \cite{KY} independently when $1/p+1/q=1$, and applied to the scattering theory of the Schr\"odinger operator $-\Delta+V$, where $V\in L^{p}(\re^d)$, $d/2\leq p<(d+1)/2$ is real-valued. Note that $(\ref{lim})$ for $V\in L^p(\re^d)$ for $d/2\leq p\leq (d+1)/2$ follow from $(\ref{L^plim})$ and H\"older's inequality.
Goldberg and Schlag \cite{GS} proved the  $L^p$-type limiting absorption principle for Schr\"odinger operator $-\Delta+V$ with a real-valued potential $V\in L^{r}(\re^d)\cap L^{3/2}(\re^d)$, $r>3/2$:
\begin{align*}
\sup_{\Re z\geq \l_0, 0<\pm \Im z\leq 1}\|(-\Delta+V-z )^{-1}\|_{B(L^p(\re^d), L^q(\re^d))}\leq C(\Re z)^{\frac{d}{2}(\frac{1}{p}-\frac{1}{q})-1},
\end{align*}
where $\l_0>0$, $d=3$, $p=4/3$ and $q=4$. The strategy of the proof in \cite{GS} is to replace the $L^2$-trace theorem in the proof of the classical Agmon-Kato-Kuroda theorem \cite[Theorem XIII. 33]{RS} by Stein-Tomas $L^p$-restriction theorem for the sphere \cite{T}.
Ionescu and Schlag \cite{IS} extends the result of \cite{GS} to a large class of potentials $V$, which contains $L^{p}(\re^d)$, $d/2\leq p\leq (d+1)/2$, the global Kato class potentials and some perturbations of first order operators. See also the recent works by Huang, Yao, Zheng \cite{HYZ} and Mizutani \cite{M}. Moreover, in \cite{IS}, it is also proved that existence and asymptotic completeness of the wave operators. We note that there are no positive eigenvalues of $-\Delta+V$  when $V\in L^{p}(\re^d)$, $d/2\leq p\leq (d+1)/2$ and it is false if $p>(d+1)/2$ (\cite{IJ} and \cite{KoT}).

In this paper, for a large class of operators $T(D)$ on $X^d$, we study uniform resolvent estimates, H\"older continuity of the resolvent and Carleman type inequalities for Fourier multipliers on $X^d$, where $X=\re$ or $X=\ze$. The uniform resolvent estimates for a Fourier multipliers are investigated in \cite{C2} and \cite{C3} in the duality line when $X=\re$ in order to study the Lieb-Thirring type bounds for fractional Schr\"odinger operators and Dirac operators.  One of the purpose is to prove the uniform resolvent estimates away form the duality line and to extend to the case of $X=\ze$. To prove this, we follow the argument in \cite[Appendix]{Gu} for the Laplacian on the Euclidean space, however, the argument in \cite{Gu} does not cover the general case since in the proof of \cite[Theorem 6]{Gu}, the spherical symmetry and the Stein-Tomas theorem for the sphere are crucial. Moreover, we study the scattering theory of the discrete Schr\"odinger operator, the fractional Schr\"odinger operators and the Dirac operators. We note that the limiting absorption principle for free discrete Schr\"odinger operators is studied in \cite{IK}, \cite{KM} and \cite{TT}. In \cite{KM}, the scattering theory of the discrete Schr\"odinger operators perturbed by $L^{p}$-potentials are studied for a range of $p$. In \cite{TT}, it is proved that the range of $(p,q)$ which the uniform resolvent estimate holds for the discrete Schr\"odinger operators differs from the one for the continuous Schr\"odinger operators when $d\geq 5$.

We remark that almost all results in this paper can be extended to the Lorentz space $L^{p,r}$ by real interpolation. For simplicity we do not mention this below.

Throughout this paper, we denote $X^d=\ze^d$ or $\re^d$ for an integer $d\geq 2$. We denote $\m$ by the Lebesgue measure if $X^d=\re^d$ by the counting measure if $X^d=\ze^d$. Moreover, we write $\widehat{X^d}=\re^d$ if $X^d=\re^d$ and $\widehat{X^d}=\T^d=(\re/\ze)^d$ if $X^d=\ze^d$.  We often use $[-1/2,1/2)^d\subset \re^d$ as a fundamental domain of $\T^d$.

Let $T\in C^{\infty}(\widehat{X^d},\re)$. Moreover, we assume $T\in \mathcal{S}'(\re^d)$ if $X=\re$. We denote the set of all critical values of $T$ by $\Lambda_{c}(T)$ and set $M_{\l}=\{\x\in \widehat{X^d}\mid T(\x)=\l\}$ for $\l\in \re$.  We denote the induced surface measure by $\m_{\l}$ away from the critical points of $T$. Moreover, for $I\subset \re$, we write $I_{\pm}=\{z\in \mathbb{C}\mid \Re z\in I,\, \pm\Im z\geq 0\}$.

Set 
\begin{align}\label{prange}
S_k=\{(\frac{1}{p}, \frac{1}{q})\in [0,1]\times [0,1]\mid \frac{1}{q}\leq \frac{1}{p}-\frac{1}{k+1},\, \frac{1+k}{1+2k}<\frac{1}{p},\, \frac{1}{q}<\frac{k}{1+2k}\}.
\end{align}

\begin{ass}\label{assa}
Let $U\subset \widehat{X^{d}}$ be a relativity compact open set and  $I\subset \re$ be an compact interval. Suppose $\pa_{\x}T(\x)\neq 0$ for $\x\in \bar{U}$.
The Fourier transform of the induced surface measure satisfies the following estimate: For any $\chi\in C_c^{\infty}(\widehat{X^d})$ supported in $U$, there exists $C>0$ such that
\begin{align}\label{surmes}
|\int_{M_{\l}}e^{2\pi ix\cdot\x}\chi(\x)d\m_{\l}(\x)|\leq C(1+|x|)^{-k},\,\, x\in X^d, \l\in I.
\end{align}
\end{ass}
\begin{rem}
If $\pa_{\x_d}T\neq 0$ on $\supp \chi$ and $\supp \chi$ is small enough, $(\ref{surmes})$ is rewritten as
\begin{align*}
|\int_{\widehat{X^{d-1}}}e^{2\pi i(x'\cdot\x'+x_dh_{\l}(\x'))}\chi(\x',h_{\l}(\x'))d\x'|\leq C'(1+|x|)^{-k},\,\, x\in X^d, \l\in I
\end{align*}
where $\x=(\x',\x_d)$ and $M_{\l}=\{(\x',\x_d)\in \widehat{X^{d}}\mid \x_d=h_{\l}(\x')\}$. Moreover, if $(\ref{surmes})$ holds, then there exits $N\geq 0$ such that
\begin{align*}
|\int_{\widehat{X^{d-1}}}e^{2\pi i(x'\cdot\x'+x_dh_{\l}(\x'))}b(\x')d\x'|\leq C\sum_{|\a|\leq N}\sup_{\x'\in \widehat{X^{d-1}}} |\pa_{\x'}^{\a}b(\x')|
\end{align*}
where $b\in C_c^{\infty}(\widehat{X^{d-1}})$ which is supported in $\{\x'\mid (\x',h_{\l}(\x'))\in \supp \chi\}$ and $C$ is independent of $b$.
\end{rem}

\begin{example}\label{ex}
Suppose that $M_{\l}\cap \supp \chi$ has at least $m$ nonvanishing principal curvature curvature at every point, then $(\ref{surmes})$ holds for $k=m/2$ by the stationary phase theorem.
\end{example}

Set $R_0^{\pm}(z)=(T(D)-z)^{-1}$ for $z\in \{z\in \mathbb{C}\mid \pm\Im z> 0\}$. Moreover, for a signature $\pm$, we define $\chi(D)R_0^{\pm}(\l\pm i 0)$ if $\pa_{\x}T\neq 0$ on $\supp \chi$  by the Fourier multiplier with its symbol $\chi(\x)(T(\x)-\l\pm i0)^{-1}$. For $1\leq p\leq \infty$, $L^p(X^d)$ denotes the Lebesgue space with the Lebesgue measure if $X=\re$ and with the counting measure if $X=\ze$.

Our first result is the following:
\begin{thm}\label{mainprop}
Let $T\in C^{\infty}(\widehat{X^d},\re)$ and let $I$ be a compact interval of $\re$. Suppose that $T^{-1}(I)$ is compact. Fix a signature $\pm$. Let $\chi\in C_c^{\infty}(\widehat{X^d})$. Suppose that $(\ref{surmes})$ holds for $\l\in I$ and $\supp \chi\subset U$. 
\item[$(i)$] There exists such that
\begin{align*}
\sup_{z\in I_{\pm}}\|\chi(D)R_0^{\pm}(z)\|_{B(L^{p}(X^d), L^{q}(X^d))}<\infty,
\end{align*}
for $(1/p,1/q)\in S_k$.
\item[$(ii)$] Set $k_{\d}=k-\d$ for $0<\d\leq 1$ and $\b_{\d}=(2/p-1)\d$. Then
\begin{align*}
\sup_{z,w\in I_{\pm}, |z-w|\leq 1}|z-w|^{-\b_{\d}}\|\chi(D)(R_0^{\pm}(z)-R_0^{\pm}(w))\|_{B(L^{p}(X^d), L^{p^*}(X^d))}<\infty,
\end{align*}
for $(1/p,1/p^*)\in S_{k_{\d}}$, where $p^*=p/(p-1)$.
\item[$(iii)$]
Suppose $X=\re$. Under Assumption \ref{assa}, for $(1/p,1/q)\in S_k$, there exists $C_{N,p,q}>0$ such that
\begin{align*}
\|\m_{N,\c}(x) \chi(D) u\|_{L^{q}(\re^d)\cap \mathcal{B}^*}\leq C_{N,p,q}\|\m_{N,\c}(x) (T(D)-\l)\chi(D) u\|_{L^{p}(\re^d)+\mathcal{B}}
\end{align*}
for $u\in \mathcal{S}(\re^d)$.
\end{thm}

\subsection{Applications to the fractional Schr\"odinger operators and the Dirac operators}

Let $n=2^{d/2}$ if $d$ is even and $n=2^{(d+1)/2}$ if $d$ is odd.
We define the Dirac operators on $\re^d$:
\begin{align*}
\mathcal{D}_0=\sum_{j=1}^d\a_jD_j,\,\,\mathcal{D}_1=\sum_{j=1}^d\a_j D_j+\a_{d+1},
\end{align*}
where $\a_j$ are $n\times n$ Hermitian matrix and satisfy the Clifford relations:
\begin{align*}
\a_j\a_k+\a_k\a_j=-2\d_{jk}I_{n\times n}
\end{align*}
and $D_j=\pa_{x_j}/(2\pi i)$.
Note that if we define $D_{d+1}=mI_{n\times n}$, then
\begin{align*}
\mathcal{D}_0^2=-(\sum_{j=1}^{d}I_{n\times n}D_j^2)=-\Delta\cdot I_{n\times n},\,\, \mathcal{D}_1^2=(-\Delta+1)\cdot I_{n\times n},
\end{align*}
where we denote $\Delta=(\sum_{j=1}^d\pa_{x_j}^2)/(4\pi^2)$.
In this subsection, we suppose that $T(D)$ is the one of the following operators:
\begin{align*}
T(D)=(-\Delta)^{s/2},\, T(D)=(-\Delta+1)^{s/2}-1,\, T(D)=\mathcal{D}_0,\, T(D)=\mathcal{D}_{1},
\end{align*}
where $0<s< d$. We use the convention that $s=1$ when $T(D)=\mathcal{D}_0$ or $T(D)=\mathcal{D}_1$. Moreover, we denote the product space $Z^{n}$ for a function space $Z$ by simply $Z$ when $T(D)=\mathcal{D}_0$ or $T(D)=\mathcal{D}_1$. As is noted in \cite[\S 2]{C2}, 
\begin{align*}
\L_c((-\Delta)^{s/2})=\begin{cases}
\{0\}\,\, &\text{if}\,s> 1,\\
\emptyset\,\, &\text{if}\, s\leq 1,
\end{cases}
\quad \L_c((-\Delta+1)^{s/2}-1)=\{0\},
\end{align*}
and 
\begin{align*}
\L_c(\mathcal{D}_0)=\{0\},\quad \L_c(\mathcal{D}_1)=\{-1,1\}.
\end{align*}
Moreover, $T(D)$ is self-adjoint on its domain $H^s(\re^d)$ by the elliptic regularity.

Let $Y_1, Y_2$ be Banach spaces such that
\begin{align}\label{YLap}
&(Y_1,Y_2) \in \bigcup_{(\frac{1}{p},\frac{1}{q})\in S_{\frac{d-1}{2}}}\{L^p(\re^d)\}\times \{L^q(\re^d)\},
\end{align}
if $2d/(d+1)\leq s< d$ and 
\begin{align}\label{Ydirac}
&(Y_1,Y_2) \in \bigcup_{\substack{(\frac{1}{p_1},\frac{1}{q_1})\in S_{\frac{d-1}{2}},\\ \frac{1}{p_2}-\frac{1}{q_2}\leq \frac{s}{d} }}\{L^{p_1}(\re^d)+L^{p_2}(\re^d)\}\times \{L^{q_1}(\re^d)\cap L^{q_2}(\re^d)\},
\end{align}
if $0<s<\frac{2d}{d+1}$.

A part of the following estimate is a generalization of \cite[Theorem3.1]{C2}.

\begin{thm}\label{diracth}
Let $I\subset \re\setminus \L_c(T(D))$ be a compact interval. We define $R_0^{\pm}(\l)$ for $\l\in I$ by the Fourier multiplier of the distribution $(T(\x)-(\l\pm i0))^{-1}$, where this distribution is well-defined since $T(\x)$ has no critical points in $T^{-1}(I)$.
\item[$(i)$] We have
\begin{align*}
\sup_{z\in I_{\pm}}\|R_0^{\pm}(z)\|_{B(Y_1,Y_2)}<\infty.
\end{align*}

\item[$(ii)$] Let $(Y_1,Y_2)$ be satisfying $p=q$ in $(\ref{YLap})$ if $2d/(d+1)\leq s< d$ and $p_1=q_1$ in $(\ref{Ydirac})$ if $0<s<2d/(d+1)$. Let $0<\d\leq 1$ and $\b_{\d}=(2/p-1)\d$. Then
\begin{align*}
\sup_{z,w\in I_{\pm}, |z-w|\leq 1}|z-w|^{-\b_{\d}}\|(R_0^{\pm}(z)-R_0^{\pm}(w))\|_{B(Y_1, Y_2)}<\infty.
\end{align*}


\item[$(iii)$] Let $V \in L^{(d+1)/2}(\re^d)\cap  L^{\infty}(\re^d)$. Assume $V$ is a self-adjoint matrix if $T(D)=\mathcal{D}_0$ or $\mathcal{D}_1$. Set $H_0=T(D)$ and $H=H_0+V$ denotes the unique self-adjoint extensions of $T(D)|_{C_c^{\infty}(\re^d)}$ and $T(D)+V|_{C_c^{\infty}(\re^d)}$ respectively. Then the wave operators
\begin{align*}
W_{\pm}=s-\lim_{t\to \pm \infty}e^{itH}e^{-itH_0}
\end{align*}
exist and are complete, i.e. the ranges of $W_{\pm}$ are the absolutely continuous subspace $\mathcal{H}_{\mathrm{ac}}(H)$ of $H$.

\item[$(iv)$] Let $V \in L^{(d+1)/2}(\re^d)\cap  L^{\infty}(\re^d,\re)$. Assume $s>1/2$ only when $T(D)=(-\Delta)^{s/2}$ with $2s\notin \mathbb{N}$. Then the set of nonzero eigenvalues $\s_{pp}(H)\setminus \{0\}$ is discrete in $\re\setminus \{0\}$. Moreover, each eigenvalue in $\s_{pp}(H)\setminus \{0\}$ has finite multiplicity. 
\end{thm}

\begin{rem}
$(i)$ is proved in \cite{C2} if $1/p+1/q=1$. In \cite{HYZ}, $(i)$ is proved when $T(D)=(-\Delta)^{s/2}$ for $2d/(d+1)\leq s<d$.
\end{rem}

\begin{rem}
In $(iii)$ and $(iv)$, the condition $V\in L^{\infty}(\re^d)$ is expected to be relaxed if we consider the appropriate selj-adjoint extension of $T(D)+V$. However, in order to avoid the technical difficulty, we assume $V\in L^{\infty}(\re^d)$.
\end{rem}

\begin{rem}
When $T(D)=\mathcal{D}_0$ or $T(D)=(-\Delta)^{s/2}$, by a scaling argument as in \cite[Remark 4.2]{C2}, we have the uniform bound of $R_0^{\pm}(z)$ with $z\in \mathbb{C}_{\pm}$. Even when $T(D)=\mathcal{D}_1$ or $T(D)=(-\Delta+1)^{s/2}-1$, the author expects to obtain the uniform bound of $R_0^{\pm}(z)$ with $z\in \mathbb{C}_{\pm}$ by further analysis.
\end{rem}

\begin{rem}
When $T(D)=(-\Delta)^{s/2}$ or $T(D)=(-\Delta+1)^{s/2}-1$, under the assumption of part $(iv)$, we can prove
\begin{align}\label{Scuni}
\sup_{z\in I_{\pm}}\|(H-z)^{-1}\|_{B(X,X^*)}<\infty
\end{align}
for any compact set $I\subset \re\setminus(\s_{pp}\cup\{0\})$.
In particular, the singular continuous spectrum of $T(D)$ is empty. For its proof, we may mimic the argument in \cite[Section 4]{IS}. However, when $T(D)=\mathcal{D}_0$ or $T(D)=\mathcal{D}_1$, the author do not know whether $(\ref{Scuni})$ holds or not since the difference of the outgoing resolvent and incoming resolvent is not always positive definite:
\begin{align*}
R_0^+(\l)-R_0^-(\l)=&(\mathcal{D}_0+\l)(\mathcal{R}_0^{+}(\l)-\mathcal{R}_0^{-}(\l)),\,\, \text{if}\,\,T(D)=\mathcal{D}_0,\\
R_0^+(\l)-R_0^-(\l)=&(\mathcal{D}_1+\l)(\mathcal{R}_1^{+}(\l)-\mathcal{R}_1^{-}(\l)),\,\, \text{if}\,\,T(D)=\mathcal{D}_1,
\end{align*}
where $\mathcal{R}_0^{\pm}(\l)=(-\Delta-(\l\pm i0)^2)^{-1}$ and $\mathcal{R}_1^{\pm}(\l)=(-\Delta+1-(\l\pm i0)^2)^{-1}$. See the arguments in \cite[Proof of Theorem 1.3 (d) and (e)]{IS} or \cite[Lemma 8 in the proof of Theorem XIII.33]{RS}.
\end{rem}

\begin{rem}
Under the assumption of $(iv)$, we can prove that each eigenfunction $u$ of $H$ associated with eigenvalue $\l\in \re\setminus \{0\}$ satisfies
\begin{align*}
(1+|x|)^Nu\in H^1(\re^d),\,\,  N\geq0
\end{align*}
and $N< s-1/2$ only when $T(D)=(-\Delta)^{s/2}$ with $s\notin 2\mathbb{N}$. The restriction $N< s-1/2$ when $T(D)=(-\Delta)^{s/2}$ with $s\notin 2\mathbb{N}$ is needed due to the singularity of the symbol $T(\x)=|\x|^s$ at $\x=0$.
\end{rem}

\subsection{Scattering theory for the discrete Schr\"odinger oeprators}

The scattering theory of the discrete Schr\"odinger operators is studied in \cite{KM} for the potential $V\in L^p(\ze^d)$, with $1\leq p<6/5 $ if $d=3$ and $1\leq p <3d/(2d+1)$ if $d\geq 4$. In this subsection, we extend their results to when $V\in L^p(\ze^d)$ for $1\leq p\leq d/3$ at the cost of the restriction of the dimension: $d\geq 4$.

We define the discrete Schr\"odinger operator:
\begin{align*}
H_0u(x)=-\sum_{|x-y|=1, y\in \ze^d}(u(x)-u(y)), \quad x\in \ze^d.
\end{align*}
Note that $H_0$ is a bounded self-adjoint operator on $L^2(\ze^d)$. We write
\begin{align*}
h_0(\x)=4\sum_{j=1}^d\sin ^2\pi \x_j\,\, \text{for}\,\, \x\in \T^d, \,\, H_0=h_0(D)
\end{align*}
and hence the spectrum $\s(H_0)$ of $H_0$ is equal to $[0,4d]$. Moreover, $\s_{ac}(H_0)=[0,4d]$, where $\s_{ac}(H_0)$ is the absolutely continuous spectrum of $H_0$. Set $R_0^{\pm}(z)=(H_0-z)^{-1}$ for $\pm\Im z>0$. Note that $\L_c(h_0(D))=\{4k\}_{k=0}^d$, where we recall that $\L_c(h_0(D))$ is the set of all critical values of $h_0(\x)$. Moreover, if $V\in L^p(\ze^d,\re)$ for some $1\leq p<\infty$, $H=H_0+V$ is a bounded self-adjoint operator and $\s_{ess}(H)=[0,4d]$ since $V\in L^p(\ze^d)\subset L^{\infty}(\ze^d)$ and $V(x)\to \infty$ as $|x|\to \infty$. Here $\s_{ess}(H)$ denotes the essential spectrum of $H$. 

 We define $R_0^{\pm}(\l)$ for $\l\in I$ by the Fourier multiplier of the distribution $(h_0(\x)-(\l\pm i0))^{-1}$, where this distribution is well-defined by virtue of \cite[Theoerem 1.8]{TT}. Note that we may take $\l$ as a critical value.
We recall that
\begin{align*}
\sup_{z\in \mathbb{C}\setminus \re}\|R_0^{\pm}(z)\|_{B(L^{p}(\ze^d), L^{p^*}(\ze^d))}<\infty,
\end{align*}
holds for $1\leq p\leq \frac{3d}{d+3}$ (\cite[Proposition 3.3]{TT}) and $d\geq 4$.

\begin{thm}\label{discth}
Fix a signature $\pm$ and let $d\geq 4$. 
\begin{itemize}
\setlength{\itemindent}{-5.8mm}
\item[$(i)$]  Let $1\leq p\leq \frac{3d}{d+3}$. Then 
\begin{align*}
\sup_{z\in \mathbb{C}_{\pm}}\|R_0^{\pm}(z)\|_{B(L^{p}(\ze^d), L^{p^*}(\ze^d))}<\infty.
\end{align*}

\item[$(ii)$] Let $1\leq p<\frac{3d}{d+3} $. Take $0<\d\leq 1$ such that $p<2/(3\d/d+(d+3)/d)$. Then
\begin{align*}
\sup_{z,w\in \mathbb{C}_{\pm}, |z-w|\leq 1}|z-w|^{-\b_{\d}}\|(R_0^{\pm}(z)-R_0^{\pm}(w))\|_{B(L^{p}(\ze^d), L^{q}(\ze^d))}<\infty.
\end{align*}

\item[$(iii)$] Let $V\in L^{p}(\ze^d)$ for $1\leq p<d/3$ and set $V^{1/2}=\sgn V|V|^{1/2}$. Then, a map $z\in I_{\pm}\mapsto |V|^{1/2}R_{0}^{\pm}(z)|V|^{1/2}$ is H\"older continuous. Moreover, for $V\in L^{d/3}(\ze^d)$, it follows that a map $z\in I_{\pm}\mapsto |V|^{1/2}R_{0}^{\pm}(z)|V|^{1/2}$ is continuous.

\item[$(iv)$] Let $V\in L^{d/3}(\ze^d,\re)$ and set $H=H_0+V$. Then the wave operators
\begin{align*}
W_{\pm}=s-\lim_{t\to \pm \infty}e^{itH}e^{-itH_0}
\end{align*}
exist and are complete, i.e. the ranges of $W_{\pm}$ are the absolutely continuous subspace $\mathcal{H}_{\mathrm{ac}}(H)$ of $H$.

\end{itemize}
\end{thm}

\begin{rem}
In Proposition \ref{disclow}, we prove that the range of $p$ can be extended in the low energy or the high energy. 
\end{rem}

We fix some notations.
For an integer $k\geq1$, $C_c^{\infty}(X^k)$ denotes $C_c^{\infty}(\re^k)$ if $X=\re$ and the set of all finitely supported functions if $X=\ze$.
For $1\leq p\leq \infty$, we write $p^*=p/(p-1)$.
We denote $t_{+}=\max{(t,0)}$ for $t\in \re$. We define the Bezov space $\mathcal{B}$ and $\mathcal{B}^*$ by
\begin{align*}
&\|u\|_{\mathcal{B}}=\|u\|_{L^2(|x|\leq 1)}+\sum_{j=1}^{\infty}2^{j/2}\|u\|_{L^2(2^{j-1}\leq |x|<2^{j})},\\
&\|u\|_{\mathcal{B}^*}=\|u\|_{L^2(|x|\leq 1)}+\sup_{j\geq 1}2^{-j/2}\|u\|_{L^2(2^{j-1}\leq |x|<2^{j})},\\
&\mathcal{B}=\{u\in L^2_{loc}(X^d) \mid \|u\|_{\mathcal{B}}<\infty\},\,\, \mathcal{B}^*=\{u\in L^2_{loc}(X^{d})\mid \|u\|_{\mathcal{B}^*}<\infty\},\\
&\mathcal{B}^*_0=\{u\in \mathcal{B}^*\mid \limsup_{R\to \infty}\frac{1}{R}\int_{|x|\leq R}|u(x)|^2dx=0 \}.
\end{align*}

\textbf{Acknowledgment.}  
The author was supported by JSPS Research Fellowship for Young Scientists, KAKENHI Grant Number 17J04478 and the program FMSP at the Graduate School of Mathematics Sciences, the University of Tokyo. The author would like to thank his supervisors Kenichi Ito and Shu Nakamura for encouraging to write this paper. The author also would like to gratefully thank Haruya Mizutani and Yukihide Tadano for helpful discussions. Moreover, the author would be appreciate Evgeny Korotyaev informing the paper \cite{KM} and Jean-Claude Cuenin for pointing out a mistake of the first draft.

\section{Abstract theorem}
In this section, we state abstract theorems which give estimates for some integral operators.
Let $K\in L^{\infty}(X^d\times X^d)$. For $x,y\in X^d$, we denote 
\begin{align*}
K(x,y)=K(x',y',x_d,y_d)=K_{x_d,y_d}(x',y'),\,\, x=(x',x_d),\,\, y=(y',y_d),
\end{align*}
where $x', y'\in X^{d-1}$ and $x_d,y_d\in X$. Moreover, we denote
\begin{align*}
Kf(x)=\int_{X^d}K(x,y)f(y)dy,\,\, T_{x_d,y_d}g(x')=\int_{X^{d-1}}K_{x_d,y_d}(x',y')f(y')dy'
\end{align*}
for $f\in C_c^{\infty}(X^d)$ and $g\in C_c^{\infty}(X^{d-1})$.

\subsection{Estimates for integral operators on duality line}

We consider the following assumptions:
\begin{ass}\label{assb}
There exists $C_0, C_1>0$ such that for any $x_d,y_d\in X$ and $g\in C_c^{\infty}(X^{d-1})$
\begin{align}
&\|T_{x_d,y_d}g\|_{L^2(X^{d-1})}\leq C_0\|g\|_{L^2(X^{d-1})},\label{assb1} \\
&\|T_{x_d,y_d} g\|_{L^{\infty}(X^{d-1})}\leq C_1(1+|x_d-y_d|)^{-k}\|g\|_{L^1(X^{d-1})}. \label{assb2}
\end{align}
\end{ass}

\begin{rem}
Suppose that we can write $K(x,y)=K_1(x'-y',x_d,y_d)$ for some $K_1\in L^{\infty}(X^{d+1})$. Then Assumption \ref{assb} directly follows from the following estimates:
\begin{align*}
&\|\int_{X^{d-1}}K_1(x',x_d,y_d)e^{-2\pi ix'\cdot \x'}dx'\|_{L^{\infty}(\widehat{X^{d-1}_{\x'}})}\leq C_0,\\
&\sup_{x'\in X^{d-1}}|K_1(x',x_d,y_d)|\leq C_1(1+|x_d-y_d|)^{-k}.
\end{align*}

\end{rem}

\begin{rem}
By the Riesz-Thorin interpolation theorem, $(\ref{assb1})$ and $(\ref{assb2})$ imply
\begin{align}
\|T_{x_d,y_d} g\|_{L^{p^*}(X^{d-1})}\leq C_0^{2-\frac{2}{p}} C_1^{\frac{2}{p}-1}(1+|x_d-y_d|)^{-k(\frac{2}{p}-1)}\|g\|_{L^p(X^{d-1})},\label{assb3}
\end{align}
for $1\leq p\leq 2$.
\end{rem}

\begin{prop}\label{abpr1}
Suppose Assumption $\ref{assb}$. Then there exists a universal constant $M_d>0$ and $M_{p,k}>0$  such that
\begin{align}
(\sup_{R>0, x_0\in \re^d}\frac{1}{R}\int_{|x-x_0|\leq R}|K f(x)|^2dx)^{\frac{1}{2}} \leq M_dC_0\|f\|_{\mathcal{B}}&,\,\, f\in \mathcal{B},\label{int1}\\
\|K f\|_{L^{p^*}(X^d)}\leq M_{p,k}C_0^{2-\frac{2}{p}} C_1^{\frac{2}{p}-1}\|f\|_{L^p(X^d)}&,\,\, f\in L^p(X^d)\label{int2}
\end{align}
for $1\leq p\leq 2(k+1)/(k+2)$.
\end{prop}

\begin{rem}
$(\ref{int2})$ follows from Proposition \ref{away} below under the assumption of Proposition \ref{away}. However, the proof below is simpler than the proof of Proposition \ref{away}.
\end{rem}

\begin{proof}
By a density argument, we may assume $f\in C_c^{\infty}(X^d)$. We observe
\begin{align}
\sup_{R>0,x_0\in X^d}\frac{1}{R}\int_{|x-x_0|<R}|Kf(x)|^2dx\leq& \sup_{x_d\in \re}\|Kf(\cdot, x_d)\|_{L^2(X^{d-1})}^2, \label{Bezc1}\\
\int_{\re}\|Kf(\cdot,y_d)\|_{L^2(X^{d-1})}dy_d\leq& M_d\|Kf\|_{B}, \label{Bezc2}
\end{align}
with some universal constant $M_d>0$. Using the Minkowski inequality and $(\ref{assb1})$, we obtain $(\ref{int1})$.

Next, we prove $(\ref{int2})$. We set $L_p=C_0^{2-\frac{2}{p}} C_1^{\frac{2}{p}-1}$. By the Minkowski inequality and $(\ref{assb3})$, we have
\begin{align*}
\|K f\|_{L^{p^*}(X^d)}=&\| \|\int_{X} T_{x_d,y_d}(f(\cdot, y_d)) dy_d\|_{L^{p^*}(X^{d-1}_{x'})}\|_{L^{p^*}(X_{x_d})}\\
\leq&L_p\| \int_{X} (1+|x_d-y_d|)^{-k(\frac{2}{p}-1)} \|f(\cdot, y_d) \|_{L^{p^*}(X^{d-1}_{y'})}dy_d\|_{L^{p^*}(X_{x_d})}\\
\leq&M_{p,k}L_p\|f\|_{L^p(X^d)},
\end{align*}
where we use the fractional integration theorem in the last line. This gives $(\ref{int2})$.
\end{proof}

\subsection{Estimates for integral operators away from duality line}

For $x_d\in X$, we define $T_{x_d}$ and $T_{x_d}^*$ by
\begin{align*}
T_{x_d}f(x')=Kf(x',x_d)=\int_{X^{d}}K(x,y)f(y)dy,\,\, T_{x_d}^*g(y)=\int_{X^{d-1}}\bar{K}(x,y)g(x')dx'.
\end{align*}
We define
\begin{align*}
S_{x_d}(y_d,z_d)g(y')=&\int_{X^{d-1}}\int_{X^{d-1}}\bar{K}(x,y)K(x,z)g(z')dz'dx'.
\end{align*}
Note that
\begin{align*}
T_{x_d}^*T_{x_d}f(y)=&\int_{X}(S_{x_d}(y_d,z_d)f(\cdot, z_d))(y')dz_d.
\end{align*}

Next, we consider the following assumption.

\begin{ass}\label{assc}
There exists $C_2, C_3>0$ such that for any $x_d, y_d, z_d\in X$
\begin{align}
&\|S_{x_d}(y_d,z_d) g\|_{L^2(X^{d-1})}\leq C_2^2\|g\|_{L^2(X^{d-1})},\label{assc1} \\
&\|S_{x_d}(y_d,z_d) g\|_{L^{\infty}(X^d)}\leq C_3^2(1+|y_d-z_d|)^{-k}\|g\|_{L^1(X^{d-1})}. \label{assc2}
\end{align}
\end{ass}

\begin{rem}
Suppose that we can write $K(x,y)=K_1(x'-y',x_d,y_d)$ for some $K_1\in L^{\infty}(X^{d+1})$. Then Assumption \ref{assc} directly follows from the following estimates:
\begin{align*}
&\|\int_{X^{d-1}}\int_{X^{d-1}}e^{2\pi iy'\cdot \x'}\bar{K}_1(x',x_d,y_d)K_1(x'-y', x_d,z_d)dx'dy'\|_{L^{\infty}(\widehat{X^{d-1}})}\leq C_2^2, \\
&\sup_{y',z'\in X^{d-1}}|\int_{X^{d-1}}\bar{K}_1(x'-y',x_d, y_d)K_1(x'-z',x_d, z_d)dx'|\leq C_3^2(1+|y_d-z_d|)^{-k}.
\end{align*}
\end{rem}

\begin{rem}
By the Riesz-Thorin interpolation theorem, $(\ref{assc1})$ and $(\ref{assc2})$ imply
\begin{align}
\|S_{x_d}(y_d,z_d) g\|_{L^{p^*}(X^{d-1})}\leq (C_2^{2-\frac{2}{p}} C_3^{\frac{2}{p}-1})^2(1+|y_d-z_d|)^{-k(\frac{2}{p}-1)}\|g\|_{L^p(X^{d-1})},\label{assc3}
\end{align}
for $1\leq p\leq 2$.
\end{rem}

\begin{prop}\label{abpr2}
Suppose that $K$ satisfies Assumption $\ref{assc}$. Then there exists a universal constant $M_{p,k}'>0$  such that
\begin{align}
(\sup_{R>0, x_0\in \re^d}\frac{1}{R}\int_{|x-x_0|\leq R}|Kf(x)|^2dx)^{\frac{1}{2}} \leq M_{p,k}'C_2^{2-\frac{2}{p}} C_3^{\frac{2}{p}-1}\|f\|_{L^p(X^d)}&,\,\, f\in L^p(X^d),\label{int3}
\end{align}
for $1\leq p\leq 2(k+1)/(k+2)$. Moreover, if $K^*(x,y)=\bar{K}(y,x)$ satisfies Assumption $\ref{assc}$, then it follows that
\begin{align}
\|K^*f\|_{L^{q}(X^d)}\leq M_{q/(q-1),k}'C_2^{\frac{2}{q}} C_3^{1-\frac{2}{q}}\|f\|_{\mathcal{B}},\,\, f\in \mathcal{B}, \label{int32}
\end{align}
for $2(k+1)/k \leq q\leq \infty$.
\end{prop}
\begin{proof}
By a density argument, we may assume $f\in C_c^{\infty}(X^d)$. First, we prove $(\ref{int3})$. Due to $(\ref{Bezc1})$, it suffices to prove
\begin{align}\label{int31}
\|T_{x_d}f\|_{L^2(X^{d-1})}\leq M_{p,k}'C_2^{2-\frac{2}{p}} C_3^{\frac{2}{p}-1}\|f\|_{L^p(X^d)},\,\, f\in C_c^{\infty}(X^d).
\end{align}
By the standard $T^*T$ argument, this estimate is equivalent to
\begin{align*}
\|T_{x_d}^*T_{x_d}f\|_{L^{p^*}(X^d)}\leq (M_{p,k}'C_2^{2-\frac{2}{p}} C_3^{\frac{2}{p}-1})^2\|f\|_{L^p(X^d)}.
\end{align*}
We set $L_p=(C_2^{2-\frac{2}{p}} C_3^{\frac{2}{p}-1})^2$. Using the Minkowski inequality and $(\ref{assc3})$, we have
\begin{align*}
\|T_{x_d}^*T_{x_d} f\|_{L^{p^*}(X^d)}=&\| \|\int_{X}(S_{x_d}(y_d,z_d)f(\cdot, z_d))(y')dz_d\|_{L^{p^*}(X^{d-1}_{y'})}\|_{L^{p^*}(X_{y_d})}\\
\leq&L_p \| \int_{X} (1+|y_d-z_d|)^{-k(\frac{2}{p}-1)} \|f(\cdot, y_d) \|_{L^{p^*}(X^{d-1}_{y'})}dy_d\|_{L^{p^*}(X_{y_d})}\\
\leq&(M_{p,k}')^2L_p\|f\|_{L^p(X^d)},
\end{align*}
where we use the fractional integration theorem (the Hardy-Littlewood-Sobolev theorem) in the last line. This proves $(\ref{int3})$.

Next, we prove $(\ref{int32})$. Replacing $K$ in $(\ref{int31})$ by $K^*$, we have
\begin{align*}
\|\int_{X^d}\bar{K}(y,x)f(y)dy\|_{L^2(X^{d-1}_{x'})}\leq M_{p,k}'C_2^{2-\frac{2}{p}} C_3^{\frac{2}{p}-1}\|f\|_{L^p(X^d)},\,\, f\in C_c^{\infty}(X^d).
\end{align*}
By duality, we have
\begin{align*}
\|\int_{X^{d-1}}K(y,x)g(x')dx'\|_{L^q(X^d_y)}\leq M_{q/(q-1),k}'C_2^{\frac{2}{q}} C_3^{1-\frac{2}{q}}\|g\|_{L^2(X^{d-1})},\,\, x_d\in X,
\end{align*}
where $q=p^*$. By $(\ref{Bezc2})$ and the Minkowski inequality, we obtain
\begin{align*}
\|Kf\|_{L^q(X^d)}\leq& \int_{X}\|\int_{X^{d-1}}K(x,y)f(y) dy'\|_{L^q(X^d_x)} dy_d\\
\leq&M_{q/(q-1),k}'C_2^{\frac{2}{q}} C_3^{1-\frac{2}{q}}\int_{X}\|f(\cdot, y_d)\|_{L^2(X^{d-1}_{y'})} dy_d\\
\leq&M_{q/(q-1),k}'C_2^{\frac{2}{q}} C_3^{1-\frac{2}{q}}\|f\|_{\mathcal{B}}.
\end{align*}

\end{proof}

We impose the additional assumption.

\begin{ass}\label{assd}
There exists $C_4>0$ such that
\begin{align*}
|K(x,y)|\leq C_4(1+|x-y|)^{-k},\,\, x\in X^d.
\end{align*}

\end{ass}

Under Assumption $\ref{assc}$ and $\ref{assd}$, we obtain the estimates similar to $(\ref{int2})$ away from the H\"older exponent.

\begin{prop}\label{away}
Suppose that $K$ and $K^*(x,y)=\bar{K}(y,x)$ satisfy Assumption $\ref{assc}$ and $\ref{assd}$. Then there exists a universal constant $L_{p,,q,k}'>0$ such that
\begin{align*}
\|K f\|_{L^{q}(X^d)}\leq L_{p,q,k}'C_{p,q,k,l}\|f\|_{L^p(X^d)}&,\,\, f\in L^p(X^d),
\end{align*}
where $1/p-1/q=1/l$ and
\begin{align*}
C_{p,q,k,l}=\begin{cases}
C_2^{\frac{2}{p^*}}C_3^{\frac{2}{p}-1}C_4^{1-\frac{2}{q}},\,\, \text{if}\,\, 1\leq p\leq \frac{(k+1)(2k+1)}{k^2+3k+1}, q>\frac{1+2k}{k}, \frac{k+1}{p^*k}\leq \frac{1}{q},\\
C_2^{\frac{2(k+1)}{2k+1}(1-\frac{1}{l}) }C_3^{\frac{2(k+1)-l}{(2k+1)l}}C_4^{\frac{l+2k}{(2k+1)l}},  \,\,  \text{if}\,\,  1\leq l\leq k+1, \frac{k}{(k+1)q}< \frac{1}{p^*}< \frac{k+1}{kq},  \\
C_2^{\frac{2}{q}}C_3^{\frac{2}{q^*}-1}C_4^{1-\frac{2}{p^*}},\,\, \text{if}\,\, 1\leq p<\frac{1+2k}{1+k}, q\geq \frac{(2k+1)(k+1)}{k^2}, \frac{k+1}{kq}\leq \frac{1}{p^*}.
\end{cases}
\end{align*}

\end{prop}

We prove this proposition by a series of lemmas.

\begin{lem}\label{R2}
Suppose that $K$ satisfies Assumption $\ref{assc}$. Let $\g\in C_c^{\infty}(\re^2)$. Define $K^j(x,y)=\g((2x_d-z_d)/2^{j+1}, (2y_d-z_d)/2^{j+1})K(x-y)$ for $j$ and $z_d\in X$. Then for $1\leq p\leq 2(k+1)/(k+2)$
\begin{align*}
\|K^jf\|_{L^2(X^d)}\leq L'M_{p,k}'C_2^{2-\frac{2}{p}} C_3^{\frac{2}{p}-1}2^{j/2}\|\g\|_{L^{\infty}(X^2)} \|f\|_{L^p(X^d)}
\end{align*}
with $L'>0$ independent of $z_d\in X$ and $j$.
\end{lem}
\begin{proof}
We take $L>0$ such that $\supp\g\subset B_L$, where $B_L\subset X^2$ is an open ball with radius $L$ and with center $0$. We observe
\begin{align*}
\|K^jf\|_{L^2(X^d)}^2=\int_{|x_d-z_d/2|\leq L2^j} \|K^jf(\cdot,x_d)\|_{L^2(X^{d-1})}^2 dx_d
\end{align*}
Replacing $K$ in $(\ref{int31})$ with $K^j$, we have
\begin{align*}
\|K^jf(\cdot,x_d)\|_{L^2(X^{d-1})}\leq M_{p,k}'C_2^{2-\frac{2}{p}} C_3^{\frac{2}{p}-1}\|\g\|_{L^{\infty}(X^2)}\|f\|_{L^p(X^d)}.
\end{align*}
We note that there exists $L'>0$ independent of $z_d$ and $j$ such that 
\begin{align*}
(\int_{|x_d-z_d/2|\leq L2^j}dx_d)^{1/2}\leq L'2^{j/2}.
\end{align*}
Combining the above three inequality, we obtain the desired result.
\end{proof}

We need the following technical lemma in order to prove Lemma \ref{techconv} below.
 
\begin{lem}\label{conte}
Let $F\in C_c^{\infty}(\re)$. Then there exists $\g\in C_c^{\infty}(\re^2)$ such that
\begin{align*}
F(\frac{x_d-y_d}{2^j})=&L_j\int_{X}\g(\frac{2x_d-z_d}{2^j}, \frac{2y_d-z_d}{2^j})dz_d,\quad x_d,y_d\in \re,
\end{align*}
where $L_j=2^{-j}$ if $X=\re$ and $2^{-j-2}\leq L_j\leq 2^{-j}$ if $X=\ze$.
\end{lem}
\begin{proof}

We define $\g\in C_c^{\infty}(\re^2)$ as follows:
Take $\chi_2\in C_c^{\infty}(\re, [0,1])$ such that $\int_{\re}\chi_2(x)dx=2$ and $\supp \chi_2\subset (-1/2,1/2)$ if $X=\re$ and such that $\chi_2(t)=1$ on $|t|\leq 1$ and $\chi(t)=0$ on $|t|\geq 2$ if $X=\ze$. We define $\g(z,z')=F(z-z')\chi_2(z+z')$, Then  $F(x_d)=\int_{X}\g(x_d+z,z)dz$ if $X=\re$ and 
\begin{align*}
\int_{X}\g(\frac{2x_d-z_d}{2^{j+1}}, \frac{2y_d-z_d}{2^{j+1}})dz_d=&\sum_{z_d\in \ze}\g(\frac{2x_d-z_d}{2^{j+1}}, \frac{2y_d-z_d}{2^{j+1}})\\
=&F(\frac{x_d-y_d}{2^j})\sum_{z_d\in \ze}\chi_2(\frac{z_d}{2^j}).
\end{align*}
if $X=\ze$. We note
\begin{align*}
2^j \leq \sum_{z_d\in \ze}\chi_2(\frac{z_d}{2^j})\leq 2^{j+2}.
\end{align*}
We set $L_j=1$ if $X=\re$ and $L_j=\sum_{z_d\in \ze}\chi_2(\frac{z_d}{2^j})$ if $X=\ze$ and we are done.
\end{proof}

The following lemma is a consequence of Lemma \ref{R2}, however its proof is a bit technical due to the convolution type cut-off. The conclusion of the following lemma is same as \cite[Lemma 1]{Gu}, where the uniform resolvent estimate of the Laplacian is studied. However, since their proof strongly depends on the spherical symmetry of the Laplacian and the Stein-Tomas theorem for the sphere, we cannot directly apply their argument to our cases. In order to overcome this difficulty,  we borrow an idea from the proof of the Carleson-Sj\"olin theorem \cite[Theorem 2.1]{H}.

\begin{lem}\label{techconv}
Suppose that $K$ satisfies Assumption $\ref{assc}$. Let $F\in C_c^{\infty}(\re)$.  Define $K^{j,\mathrm{conv}}(x,y)=F((x_d-y_d)/2^j)K(x,y)$ for non-negative integer $j$. Then for $1\leq p\leq 2(k+1)/(k+2)$, there exists a universal constant $M_{p,k}''$ such that
\begin{align}\label{l^p}
\|K^{j,\mathrm{conv}} f(x)\|_{L^2(X^d)}\leq M_{p,k}''C_2^{2-\frac{2}{p}} C_3^{\frac{2}{p}-1}2^{\frac{1}{2}j} \|f\|_{L^p(X^d)}
\end{align}
\end{lem}

\begin{proof}

By Lemma \ref{conte}, we have
\begin{align*}
|K^{j,\mathrm{conv}}f(x) |\leq 2^{-2j}|\int_{X}K^{j,z_d}(x,y)dz_d |,
\end{align*}
where we set $K^{j,z_d}(x,y)=K(x,y)\g((2x_d-z_d)/2^{j+1}, (2y_d-z_d)/2^{j+1})$.
Take $\f\in C_c^{\infty}(\re)$ such that $\g(x_d,y_d)=\g(x_d,y_d)\f(y_d)$.  We take $L>0$ such that $\supp\g\subset B_L$, where $B_L\subset X^2$ is an open ball with radius $L$ and with center $0$. We note
\begin{align*}
|\{z_d\in X\mid \g(\frac{2x_d-z_d}{2^{j+1}},\frac{2y_d-z_d}{2^{j+1}} )\neq 0\}|\leq L2^{j+1}.
\end{align*}
Set $M=(M_{p,k}'C_2^{2-\frac{2}{p}}C_3^{\frac{2}{p}-1})^2$. Using the Cauchy-Schwarz inequality and Lemma \ref{R2}, we have
\begin{align*}
\int_{X^d}|\int_{X}K^{j,z_d}f(x)dz_d |^2dx\leq& L2^{j+1} \int_{X}\int_{X^d}|K^{j,z_d}f(x) |^2dxdz_d\\
\leq& 2LL'M 2^{2j} \int_{X}\|\f(\frac{2\cdot-z_d}{2^{j+1}})f\|_{L^p(X^d)}^2  dz_d.
\end{align*}
Since $p\leq 2$, by using the Minkowski inequality, we have
\begin{align*}
\int_{X}\|\f(\frac{2\cdot-z_d}{2^{j+1}})f\|_{L^p(X^d)}^2  dz_d\leq& \|\f(\frac{2\cdot-z_d}{2^{j+1}})\|_{L^2(X)}^2 \|f\|_{L^p(X^d)}^2\\
\leq &L''2^{j}\|\f\|_{L^2(X)}^2 \|f\|_{L^p(X^d)}^2
\end{align*}
with $L''$ depends only on $\f$.
Thus we obtain
\begin{align*}
\int_{X^d}|K^{j,conv}f(x)|^2dx\leq (M_{p,k}''C_2^{2-\frac{2}{p}}C_3^{\frac{2}{p}-1})^22^j\|f\|_{L^p(X^d)}^2,
\end{align*}
where $(M_{p,k}'')^2=2LL'L''(M_{p,k}')^2\|\f\|_{L^2(X)}^2$.
\end{proof}

\begin{cor}\label{abcor}
Suppose that $K$ satisfies Assumption \ref{assd}. Then there exists a constant $L_1>0$ which depends only on $F$, $d$ and $k$ such that
\begin{align}\label{l^1}
\|K^{j,conv} f\|_{L^{\infty}(X^d)}\leq L_1C_42^{-jk}\|f\|_{L^1(X^d)}.
\end{align}
In addition, we suppose that $K$ and $K^*(x,y)=\bar{K}(y,x)$ satisfy Assumption \ref{assc}. Set $1/p_1=1-q/2p^*$ and $L_{2,p,q}=(M_{p_1,k}'')^{2/q}L_1^{1-2/q}$. Then
\begin{align}\label{l^p1}
\|K^{j,conv} f\|_{L^q(X^d)}\leq L_{2,p,q}C_2^{\frac{2}{p^*}}C_3^{\frac{2}{p}-1}C_4^{1-\frac{2}{q}}2^{j\frac{1+2k}{q}-jk} \|f\|_{L^p(X^d)}
\end{align}
if $q\geq 2$ and $(k+1)(1-1/p)/k\leq 1/q$ and
\begin{align}\label{l^p2}
\|K^{j,conv} f\|_{L^q(X^d)}\leq L_{2,q^*,p^*}C_2^{\frac{2}{q}}C_3^{\frac{2}{q^*}-1}C_4^{1-\frac{2}{p^*}}2^{j\frac{(1+2k)}{p^*}-jk} \|f\|_{L^p(X^d)}
\end{align}
if $p\leq 2$ and $(k+1)/(kq)\leq 1-1/p$.

\end{cor}
\begin{proof}
$(\ref{l^1})$ follows from
\begin{align*}
\|K^{j,conv} f\|_{L^{\infty}(X^d)}\leq& \|F(\cdot/2^j) K\|_{L^{\infty}(X^d)}\|f\|_{L^1(X^d)}\\
\leq&L_1C_42^{-jk}\|f\|_{L^1(X^d)}
\end{align*}
with some constant $L_1>0$ by Assumption \ref{assd}.
By complex interpolating $(\ref{l^p})$ and $(\ref{l^1})$, we obtain $(\ref{l^p1})$. Since $K^*$ also satisfies Assumption \ref{assc} and \ref{assd}, by duality, $(\ref{l^p2})$ holds.

\end{proof}

\begin{proof}[Proof of Proposition \ref{away}]
Take $\y\in C_c^{\infty}(\re,[0,1])$ such that $\y(t)=1$ on $0\leq t\leq 1$ and $\y=0$ on $t\geq 2$. Set $F(x)=\y(|x|)-\y(|x|/2)$. 
By Corollary \ref{abcor}, for $(k+1)(1-1/p)/k\leq 1/q$, $q>(1+2k)/k$, we have
\begin{align*}
\|K f\|_{L^q(X^d)}=&\|\sum_{j=0}^{\infty}K^{j, conv} f\|_{L^q(X^d)}
\leq\sum_{j=0}^{\infty}\|K^{j, conv} f\|_{L^q(X^d)}\\
\leq&L_{2,p,q}C_2^{\frac{2}{p^*}}C_3^{\frac{2}{p}-1}C_4^{1-\frac{2}{q}}\sum_{j=0}^{\infty}2^{j/2+jd(1/q-1/2)} \|f\|_{L^p(X^d)}\\
\leq&L_{2,p,q}'C_2^{\frac{2}{p^*}}C_3^{\frac{2}{p}-1}C_4^{1-\frac{2}{q}}\|f\|_{L^p(X^d)},
\end{align*}
where $L_{2,p,q}'=L_{2,p,q}\sum_{j=0}^{\infty}2^{j/2+jd(1/q-1/2)}$. Similarly, for $(k+1)/(kq)\leq 1-1/p$, $p<(1+2k)/(1+k)$, we have
\begin{align*}
\|K f\|_{L^q(X^d)}\leq L_{2,q^*,p^*}'C_2^{\frac{2}{p^*}}C_3^{\frac{2}{p}-1}C_4^{1-\frac{2}{p^*}}  \|f\|_{L^p(X^d)}.
\end{align*}

In order to prove the end point estimates, we use Bourgain's interpolation trick (\cite{B}, \cite[\S 6.2]{CSWW}, \cite[Lemma 3.3]{JKL}). This trick is also used in \cite{BS} for the Stein-Tomas theorem for a large class of measures in Euclidean space. See also \cite{FSc} and \cite{Gu}.
We denote the Lorentz space for index $1\leq p\leq \infty$ and $1\leq r\leq \infty$ by $L^{p,r}(X^d)$:
\begin{align*}
&\|f\|_{L^{p,r}(X^d)}=\begin{cases}
p^{\frac{1}{r}}(\int_0^{\infty}\m(\{x\in X^d\,|\, |f(x)|>\a\})^{\frac{r}{p}} \a^{r-1}d\a)^{\frac{1}{r}},\quad &r<\infty,\\
\sup_{\a>0}\a\m(\{x\in X^d\mid |f(x)|>\a\})^{\frac{1}{p}},\quad &r=\infty,
\end{cases}\\
&L^{p,r}(X^d)=\{f:X^d\to \mathbb{C}\mid f:\text{measurable},\, \|f\|_{L^{p,r}(X^d)}<\infty\}.
\end{align*}

Bourgain's interpolation trick with $(\ref{l^p1})$ and $(\ref{l^p2})$ implies that for $1\leq p\leq (k+1)(2k+1)/(k^2+3k+1), q=(1+2k)/k$, it follows that
\begin{align*}
\|K f\|_{L^{q,\infty}(X^d)}\leq& L_{2,p,q}'C_2^{\frac{2}{p^*}}C_3^{\frac{2}{p}-1}C_4^{1-\frac{2}{q}}\|f\|_{L^{p,1}(X^d)}
\end{align*}
with a universal constant $L_{2,p,q}'$. Similarly, for $p=(1+2k)/(1+k)$, $q\geq (2k+1)(k+1)/k^2$, we have
\begin{align*}
\|K f\|_{L^{q,\infty}(X^d)}\leq L_{2,q^*,p^*}'C_2^{\frac{2}{q}}C_3^{\frac{2}{q^*}-1}C_4^{1-\frac{2}{p^*}}  \|f\|_{L^{p,1}(X^d)}.
\end{align*}
By real interpolating above estimates, we complete the proof.

\end{proof}

\section{Uniform resolvent estimates}

\subsection{Proof of Theorem \ref{mainprop} $(i)$ and $(ii)$}
\begin{proof}[Proof of Theorem \ref{mainprop} $(i)$ and $(ii)$]
We follows the argument as in \cite[Lemma 3.3]{C2}.
By using a partition of unity and a linear coordinate change, we may assume that $\pa_{\x_d}T\neq 0$ on $\supp \chi$. Moreover, by the implicit function theorem, we may assume that for $\l\in I$, $M_{\l}$ has the following graph representation:
\begin{align*}
M_{\l}\cap \supp\chi \subset \{(\x', h_{\l}(\x')) \in \widehat{X^{d}} \mid \x'\in U\}
\end{align*}
for some relativity compact open set $U\subset \widehat{X^{d-1}}$ and $h_{\l}$ which is smooth with respect to $\x'\in U$ and $\l\in I$ and
\begin{align}\label{Timp}
T(\x)-\l=e(\x,\l)(\x_d-h_{\l}(\x')),
\end{align}
where $e(\x,\l)=\int_0^1(\pa_{\x_d}T)(\x', t\x_d+(1-t)h_{\l}(\x'))dt$. Furthermore, we may assume $\min_{\x\in \supp\chi, \l\in A} e(\x,\l)>0$ if necessary, we take $\supp \chi$ small. Set 
\begin{align*}
K_{z,\pm}(x)=&\int_{\widehat{X^d}}\frac{e^{2\pi ix\cdot \x}\chi(\x)}{T(\x)-z}d\x,\\
K_{z,w,\pm}(x)=&K_{z,\pm}(x)-K_{w,\pm}(x),
\end{align*}
where $ \l=\Re z, \m=\Re w\in I$ and $\pm \Im z, \pm \Im w\geq 0$.
In order to prove Theorem \ref{mainprop}, it suffices to show that $K_{z,\pm}$ satisfies Assumptions \ref{assc} and \ref{assd}, and that $K_{z,w,\pm}$ satisfies Assumption \ref{assb}.
\begin{lem}\label{unilem1}
Fix a signature $\pm$. For any $0\leq \d \leq 1$, there exists $C_0, C_1, C_{1,\d}>0$ such that for $x=(x',x_d)\in X^d$, $z,w\in I_{\pm}$ with $|z-w|\leq 1$, we have
\begin{align*}
&\sup_{\x'\in \widehat{X^{d-1}}}|\int_{X^{d-1}}K_{z,\pm}(y',x_d)e^{-2\pi iy'\cdot \x'}dy'|\leq C_0,\,\, |K_{z,\pm}(x)|\leq C_1(1+|x|)^{-k}\\
&\sup_{\x'\in \widehat{X^{d-1}}}|\int_{X^{d-1}}K_{z,w,\pm}(y',x_d)e^{-2\pi iy'\cdot \x'}dy'|\leq 2C_0,\\
&|K_{z,w,\pm}(x)|\leq C_1' |z-w|^{\d}(1+|x|)^{-k+\d}
\end{align*}
\end{lem}
\begin{proof}
Note that $\int_{X^{d-1}}K_{z,\pm}(y',x_d)e^{-2\pi iy'\cdot \x'}dy'=\int_{\widehat{X}}\frac{e^{2\pi ix_d\x_d}\chi(\x)}{T(\x)-z}d\x_d$. If necessary we take $\supp \chi$ is small, it suffices to replace the integration region by $\re$. Thus by $(\ref{Timp})$, we have
\begin{align*}
\int_{\widehat{X}}\frac{e^{2\pi ix_d\x_d}\chi(\x)}{T(\x)-z}d\x_d=&\int_{\re}\frac{e^{2\pi ix_d\x_d}\chi(\x)}{T(\x)-z}d\x_d\\
=&\int_{\re}\frac{e^{2\pi ix_d(\x_d+h_{\l}(\x'))}\chi(\x', \x_d+h_{\l}(\x'))}{e(\x', \x_d+h_{\l}(\x'), \x_d)\x_d-i \Im z}d\x_d\\
=&:e^{2\pi ix_dh_{\l}(\x')}\c_{z, \pm}(\x',x_d).
\end{align*}
By using \cite[(3.10)]{C2} for $\pm\Im z>0$ and \cite[(A.6)]{C3} for $\pm \Im z=0$, we have 
\begin{align}\label{games}
|\pa_{\x'}^{\a}\c_{z, \pm}(\x',x_d)|\leq C_{\a}
\end{align}
for $\a\in \mathbb{N}^{d-1}$. We will prove $(\ref{games})$ in Lemma \ref{appth}. Thus the first inequality holds. Moreover, we note that
\begin{align*}
K_{z,\pm}(x)=\int_{\widehat{X^{d-1}}}\c_{z, \pm}(\x')e^{2\pi i(x'\cdot \x'+ x_dh_{\l}(\x'))}d\x'.
\end{align*}
Since $\c_{z,\pm}$ is compactly supported in $\x'$-variable, then $(\ref{surmes})$ and $(\ref{games})$ imply the second inequality. The estimates for $K_{z,w,\pm}(x)$ follow from the estimates
\begin{align*}
|\pa_{\x'}^{\a}\c_{z, \pm}(\x',x_d)|\leq C_{\a}'|z-w|^{\d}(1+|x_d|)^{\d},
\end{align*}
which is also proved after Lemma \ref{appth}: $(\ref{games3})$.
\end{proof}

\begin{lem}\label{unilem2}
There exists $C_3>0$ such that
\begin{align*}
&|\int_{X^{d-1}}\int_{X^{d-1}}e^{2\pi iy'\cdot \x'}\bar{K}_{z,\pm}(x',x_d-y_d)K_{z,\pm}(x'-y', x_d-z_d)dx'dy'|\leq C_0^2\\
&|\int_{X^{d-1}}\bar{K}_{z,\pm}(x'-y',x_d-y_d)K_{z,\pm}(x'-z',x_d-z_d)dx'|\leq C_3^2(1+|y_d-z_d|)^{-k}
\end{align*}
where $C_0>0$ is as in the proof of Lemma \ref{unilem1}.
\end{lem}
\begin{proof}
Note that
\begin{align*}
&\int_{X^{d-1}}\int_{X^{d-1}}e^{2\pi iy'\cdot \x'}\bar{K}_{z,\pm}(x',x_d-y_d)K_{z,\pm}(x'-y', x_d-z_d)dx'dy'\\
&=e^{2\pi i(y_d-z_d)h_{\l}(\x')}\c_{z,\pm}(\x',x_d-z_d)\overline{\c_{z, \pm}(\x',x_d-y_d)},
\end{align*}
where $\c_{z,\pm}$ is as in the proof of Lemma \ref{unilem1}. Moreover, we have
\begin{align*}
&\int_{X^{d-1}}\bar{K}_{z,\pm}(x'-y',x_d-y_d)K_{z,\pm}(x'-z',x_d-z_d)dx'\\
&=\int_{\widehat{X}^{d-1}}e^{2\pi i(y'-z')\cdot\x'+2\pi i(y_d-z_d)h_{\l}(\x')}\c_{z,\pm}(\x',x_d-z_d)\overline{\c_{z, \pm}(\x',x_d-y_d)}d\x'.
\end{align*}
Thus $(\ref{surmes})$ and $(\ref{games})$ imply the conclusion.

\end{proof}

Lemma \ref{unilem1} and \ref{unilem2} imply that $K_{z,\pm}$ satisfies Assumptions \ref{assc} and \ref{assd} and $K_{z,w,\pm}$ satisfies Assumption \ref{assb}. This completes the proof of Theorem \ref{mainprop}.

\end{proof}

\begin{rem}
In order to prove $(i)$, it is sufficient to prove $(i)$ for $\pm \Im z=0$ by using the Phragm\'en-Lindel\"of principle as in \cite[Section 5.3]{R}. See also \cite[Appendix A]{C2} for the estimates of the Shatten norm of the resolvent. Here we avoid using the Phragm\'en-Lindel\"of principle.
\end{rem}

\begin{cor}\label{abBir}
Let $r_1,r_2\in (1,4k+2]$ satisfying $1/r_1+1/r_2\geq 1/(k+1)$. Then
\begin{align}\label{Bir1}
\sup_{z\in I_{\pm}}\|W_1\chi(D)R_0^{\pm}(z)W_2\|_{B(L^2(X^d))}\leq C\|W_1\|_{L^{r_1}(X^d)}\|W_2\|_{L^{r_2}(X^d)}
\end{align}
for $W_1\in L^{r_1}(X^d)$ and $W_2\in L^{r_2}(X^d)$. Moreover, let $W_1\in L^{r_1}(X^d)$ and $W_2\in L^{r_2}(X^d)$. Then it follows that $W_1\chi(D)R_0^{\pm}(z)W_2$ belongs to $B_{\infty}(L^2(X^d))$ and a map $z\in I_{\pm}\mapsto W_1\chi(D)R_0^{\pm}(z)W_2\in B_{\infty}(L^2(X^d))$ is continuous in $z\in I_{\pm}$.
In addition, for $r=r_1=r_2 \in (1,4k_{\d}+2)$, we have
\begin{align}\label{Bir2}
\|W_1\chi(D)(R_0^{\pm}(z)-R_0^{\pm}(w))W_2\|_{B(L^2(X^d))}\leq C|z-w|^{\b_{\d}}\|W_1\|_{L^{r}(X^d)}\|W_2\|_{L^{r}(X^d)}
\end{align}
for $z,w\in I_{\pm}, |z-w|\leq 1$.
\end{cor}

\begin{proof}
$(\ref{Bir1})$ and $(\ref{Bir2})$ follow from Theorem \ref{mainprop} and the H\"older inequality. For proving the other statements, we may assume $W_1,W_2\in C_c^{\infty}(X^d)$ by $\e/3$-argument and $(\ref{Bir1})$. Since $W_1$ and $W_2$ are compactly supported and since the integral kernel of $\chi(D)R_0^{\pm}$ is in $L^{\infty}$ by Lemma \ref{unilem1}, then the integral kernel of $W_1\chi(D)R_0^{\pm}(z)W_2$ is square integrable and hence Hilbert-Schmidt. Thus it follows that $W_1\chi(D)R_0^{\pm}(z)W_2$ is compact. Moreover, by $(\ref{Bir2})$, we see that $W_1\chi(D)R_0^{\pm}(z)W_2$ is continuous in $z\in I_{\pm}$. The case of $W_1\in L^{r_1}(X^d)$ and $W_2\in L^{r_2}(X^d)$ follows from the $\e/3$-argument as in the proof of Lemma \ref{discear}.
\end{proof}

\subsection{Supersmoothing, Proof of Theorem \ref{mainprop} $(iii)$}
In this subsection, we assume $X=\re$. The author expect that the following proposition with $X=\ze$ holds. However, we prove this with with $X=\re$ for possibly technical reason.
We recall $\m_{N,\c}(x)=(1+|x|^2)^N(1+\c|x|^2)^{-N}$. We restate Theorem \ref{mainprop} $(iii)$:

\begin{prop}\label{superprop}
Let $I\subset \re$ be a compact interval. Suppose $T^{-1}(I)$ is compact.
Let $\chi\in C_c^{\infty}(\re^d)$ be supported in $T^{-1}(I)$. Under Assumption \ref{assa}, for $(1/p,1/q)\in S_k$, there exists $C_{N,p,q}>0$ such that
\begin{align}\label{super1}
\|\m_{N,\c}(x) \chi(D) u\|_{L^{q}(\re^d)\cap \mathcal{B}^*}\leq C_{N,p,q}\|\m_{N,\c}(x) (T(D)-\l)\chi(D) u\|_{L^{p}(\re^d)+\mathcal{B}}
\end{align}
for $u\in \mathcal{S}(\re^d)$.
\end{prop}

\begin{lem}\label{patlem}
Suppose that $m\in C^{\infty}(\re^d)$ satisfies
\begin{align*}
|\pa_{\x}^{\a}m(\x)|\leq C_{\a}(1+|\x|^2)^{-|\a|/2}
\end{align*}
for $\a\in \mathbb{N}^d$.
Let $1<p<\infty$. We set $\tilde{\m}_{N,\c}(x_d)=(1+|x_d|^2)^{N}(1+\c|x|^2)^{-N}$. Then we have
\begin{align*}
&\|\m(x)m(D)\m(x)^{-1}\|_{B(L^{p}(\re^d))}\leq C_{N,m,p},\,\, \|\m(x)m(D)\m(x)^{-1}\|_{B(\mathcal{B}(\re^d))}\leq C_{N,m},\\
&\|\m(x)m(D)\m(x)^{-1}\|_{B(\mathcal{B}^*(\re^d))}\leq C_{N,m}
\end{align*}
if $\m(x)\in \{\m_{N,\c}(x), \m_{N,\c}^{-1}(x), \tilde{\m}_{N,\c}(x_d), \tilde{\m}_{N,\c}^{-1}(x_d)\}$, where $C_{N,m,p}$ and $C_{N,m}$ are independent of $0<\c\leq 1$ and depends only on $d$, $N$ and finite number of $C_M$.
\end{lem}

\begin{proof}
The proof is same as in the proof of \cite[(3.7)]{IS}. In fact, though the range of $p$ is restricted in \cite{IS}, the proof succeeds even when $1<p<\infty$.
\end{proof}

\begin{lem}\label{mulemma}
\item[$(i)$]
For $\a\in \mathbb{N}^d$, we have
\begin{align}
\pa_{x}^{\a}\m_{N,\c}(x)=&b_{\a}(x)\m_{N,\c}(x) \label{mulem1}\\
\pa_{x}^{\a}\m_{N,\c}(x)^{-1}=&b_{\a}'(x)\m_{N,\c}(x)^{-1}\label{mulem2}
\end{align}
for some functions $b_{\a}, b_{\a}'\in C^{\infty}(\re^d)$ such that for $\b\in \mathbb{N}^d$, 
\begin{align*}
|(1+|x|^2)^{(|\a|+|\b|)/2}\pa_{x}^{\b}b_{\a}(x)|\leq C_{\a,\b,N},\,\, |(1+|x|^2)^{(|\a|+|\b|)/2}\pa_{x}^{\b}b_{\a}'(x)|\leq C_{\a,\b,N}
\end{align*}
with some constant $C_{\a,\b,N}$ which is independent of $0<\c\leq 1$.
\item[$(ii)$] There exists $C_N>0$ independent of $0<\c\leq 1$ such that
\begin{align*}
\m_{N,\c}(x)\m_{N,\c}(y)^{-1}+\m_{N,\c}(y)\m_{N,\c}(x)^{-1}\leq C_N(1+|x-y|^2)^{N},\,\, x,y\in \re^d.
\end{align*}

\end{lem}

\begin{proof}
$(i)$
We prove $(\ref{mulem1})$ only. The proof of $(\ref{mulem2})$ is similar.
We prove $(\ref{mulem1})$ by induction in $|\a|$. If $\a=0$, then $(\ref{mulem1})$ is trivial. Let $M>0$ be an integer. Suppose that $(\ref{mulem1})$ holds for $|\a|\leq M$. If $|\a|=M$, by the induction hypothesis, we have
\begin{align*}
\pa_{x_j}\pa_{x}^{\a}\m_{N,\a}(x)=&(\pa_{x_j}b_{\a}(x))\m_{N,\c}(x)+b_{\a}(x)\pa_{x_j}\m_{N,c}(x)\\
=&((\pa_{x_j}b_{\a})(x)+b_{\a}(x)b_{e_j}(x))\m_{N,\c}(x),
\end{align*}
where $(e_1,...,e_d)$ is a standard basis in $\re^d$. Thus, if we set $b_{\a+e_j}(x)=(\pa_{x_j}b_{\a})(x)+b_{\a}(x)b_{e_j}(x)$, then $|(1+|x|^2)^{(|\a|+|\b|)/2}\pa_{x}^{\b}b_{\a}(x)|\leq C_{\a,\b,N}$ follows. This proves $(\ref{mulem1})$ for $|\a|=M+1$. $(ii)$ is easily proved. 
\end{proof}

\begin{cor}\label{mucor}
For $k\in \re$ we define $\L_k=(I-\Delta)^{k/2}$. Then 
\begin{align*}
\|\m\L_k\m^{-1}\L_{-k}\|_{B(L^p(\re^d))}+\|\m\L_k\m^{-1}L_{-k}\|_{B(\mathcal{B})}+\|\m\L_k\m^{-1}L_{-k}\|_{B(\mathcal{B}^*)}\leq& C_{N,k,p},\\
\|\L_k\m\L_{-k}\m^{-1}\|_{B(L^p(\re^d))}+\| \L_k\m \L_{-k}\m^{-1}\|_{B(\mathcal{B})}+\|\L_k\m \L_{-k}\m^{-1}\|_{B(\mathcal{B}^*)}\leq& C_{N,k,p},
\end{align*}
with some $C_{N,k,p}>0$ independent of $0<\c\leq 1$ for $\m\in \{\m_{N,\c}, \m_{N,\c}^{-1}\}$ and $1<p<\infty$.
\end{cor}
\begin{proof}
The proof is same as in \cite[Lemma 3.2]{IS} by virtue of Lemma \ref{patlem} and \ref{mulemma}.
\end{proof}

\begin{proof}[Proof of Proposition \ref{superprop}]
Let $Y_1\in \{L^{p}(\re^d), \mathcal{B}\}$ and $Y_2\in \{L^{q}(\re^d), \mathcal{B}^*\}$.
If necessary, we may assume $\supp \chi $ is small enough. In fact, by using a partition of unity $\{\chi_j\}_{j=1}^M$ such that $\sum_{j=1}^M\chi_j=1$ on $\supp \chi$, we have
\begin{align*}
\|\m_{N,\c}(x) \chi(D) u\|_{Y_2}\leq& \sum_{j=1}^M\|\m_{N,\c}(x) (\chi_j\chi)(D) u\|_{Y_2},\\
\sum_{j=1}^M\|\m_{N,\c}(x) (T(D)-\l)(\chi_j\chi)(D) u\|_{Y_1}\leq& C_{N,m,p}\|\m_{N,\c}(x) (T(D)-\l)\chi(D) u\|_{Y_1},
\end{align*}
where we use the triangle inequality in the first line and Lemma \ref{patlem} in the second line. Thus we may replace $\chi(D)$ by $(\chi_j\chi)(D)$ in $(\ref{super1})$.
 
We may suppose $\hat{u}$ and $\hat{f}$ are supported in $\supp \chi$ and we may suppose $\pa_{\x_d}T\neq 0$ on $\supp\chi$ by rotating the coordinate and by taking $\supp \chi$ small enough. We set $\x_j^+=\e_0 e_j+\sqrt{1-\e_0^2}e_d$ for $j=1,...,d-1$ and $\x_d^+=\x_d$, where $\e_0>0$ is a small constant and $(e_1,...,e_d)$ is the standard basis of $\re^d$. Since $(\x_1^+, ...,\x_d^+)$ is the basis of $\re^d$, then
\begin{align*}
C^{-1}\sum_{j=1}^d\tilde{\m}_{N,\c}(x\cdot \x_j^+) \leq \m_{N,\c}(x)\leq C\sum_{j=1}^d\tilde{\m}_{N,\c}(x\cdot \x_j^+)
\end{align*}
with some constant $C>0$ independent of $\c$, where 
\begin{align*}
\tilde{\m}_{N,\c}(t)=(1+t^2)^N(1+\c t^2)^{-N}. 
\end{align*}
Thus it suffices to prove that
\begin{align*}
\|\tilde{\m}_{N,\c}(x\cdot \x_j^+)u\|_{Y_2}\leq C_{N}\|\tilde{\m}_{N,\c}(x\cdot \x_j^+)(T(D)-\l)u\|_{Y_1}
\end{align*}
for each $j=1,...,d$. If $\e_0>0$ is small, then $\pa_{\x_d}T\neq 0$ implies $\x_j^+\cdot \nabla T(\x)=\e_0 \pa_{\x_1}T+\sqrt{1-\e_0^2}\pa_{\x_d}T\neq 0$ on $\supp \chi$. Thus by rotating the coordinate, we may reduce to prove 
\begin{align*}
\|\tilde{\m}_{N,\c}(x_d)u\|_{Y_2}\leq C_{N}\|\tilde{\m}_{N,\c}(x_d)(T(D)-\l)u\|_{Y_1}.
\end{align*}
We remark that this reduction is the only part to miss proving this Proposition when $X=\ze$. In fact, there are no basis containing the normal vector of $x\cdot \x_j^+$-direction when $X=\ze$.

Set $f=(T(D)-\l)u$. By the implicit function theorem, we have $T(\x)-\l=e(\x,\l)(\x_d-h_{\l}(\x'))$ as in $(\ref{Timp})$. Then we have $e(\x,\l)^{-1}\hat{f}(\x)=(\x_d- h_{\l}(\x'))\hat{u}(\x)$ on $\supp \chi$. We denote $\tilde{f}(\x',x_d)$ is the Fourier transform of $f$ with respect to $\x_1,...,\x_{d-1}$-variables and  set $\hat{g}(\x)=e(\x,\l)^{-1}\hat{f}(\x)$. Here $e(\x,\l)^{-1}$ is well-defined on $\supp \hat{f}$ since $\supp f\subset \supp \chi$. Then 
\begin{align*}
(D_{x_d}-h_{\l}(\x'))\tilde{u}(\x',x_d)=\tilde{g}(\x',x_d),
\end{align*}
Since $\tilde{u}$ and $\tilde{g}$ are smooth, by using variation of parameters, we can write
\begin{align*}
\tilde{u}(\x',x_d)=&\int_{-\infty}^{x_d}e^{2\pi i(x_d-y_d)h_{\l}(\x')}\tilde{g}(\x',y_d)dy_d\\
=&-\int_{x_d}^{\infty}e^{2\pi i(x_d-y_d)h_{\l}(\x')}\tilde{g}(\x',y_d)dy_d.
\end{align*}
Note that we use the first line of the above representation if $x_d\leq 0$ and the second line if $x_d\geq 0$.
Taking the inverse Fourier transform and multiplying $\tilde{\m}_{N,\c}(x_d)$, we have
\begin{align*}
\tilde{\m}_{N,\c}(x_d)u(x)=&\int_{\re}\int_{\re^{d-1}}K_{N,\c}(x'-y',x_d, y_d)\tilde{\m}_{N,\c}(y_d)g(y)dy'dy_d
\end{align*}
where
\begin{align*}
K_{N,\c}(x'-y',x_d,y_d)=&\frac{\tilde{\m}_{N,\c}(x_d)}{\tilde{\m}_{N,\c}(y_d)}(\chi_{x_d<0}\chi_{x_d\leq y_d}-\chi_{x_d>0}\chi_{x_d\leq y_d})\\
&\times \int_{\widehat{\re^{d-1}}}e^{2\pi i(x'-y')\cdot \x'+2\pi i(x_d-y_d)h_{\l}(\x')}\g(\x')d\x'.
\end{align*}
Note that $\frac{\tilde{\m}_{N,\c}(x_d)}{\tilde{\m}_{N,\c}(y_d)}(\chi_{x_d<0}\chi_{x_d\leq y_d}-\chi_{x_d>0}\chi_{x_d\leq y_d})\leq 1$. Let $R$ be the linear operator on $\re^d$ with the integral kernel $K_{N,\c}$. We recall $\supp\hat{f}\subset \supp \chi$ and $\hat{g}=e(\x,\l)^{-1}\hat{f}(\x)$. Hence we can write
\begin{align*}
\tilde{\m}_{N,\c}(x_d)u(x)=K_{N,\c}(x'-y')\ast (\tilde{\m}_{N,\c}(y_d)\f(D)e(D,\l)^{-1}\tilde{\m}_{N,\c}^{-1}(y_d)\tilde{\m}_{N,\c}(y_d)f)(x)
\end{align*}
where $\f\in C_c^{\infty}(\re^d)$ such that $\f=1$ on $\supp \chi$. By virtue of Lemma \ref{patlem}, it follows that the operator norms of $\tilde{\m}_{N,\c}(y_d)\chi(D)e(D,\l)^{-1}\tilde{\m}_{N,\c}(y_d)^{-1}$ on $L^p(\re^d)$ ($1<p<\infty$), $\mathcal{B}$ and $\mathcal{B}^*$ are uniformly bounded in $\l\in I$. 

By virtue of Propositions \ref{abpr2} and \ref{away}, it suffices to $K_{N,\c}$ and $K_{N,\c}^*(x,y)=\bar{K}_{N,\c}(y,x)$ satisfies Assumptions \ref{assc} and \ref{assd}. To see this, we may mimic the proof of Lemma \ref{unilem2}. We omit the detail.

\end{proof}

\section{Applications}

\subsection{Fractional Schr\"odinger operators and Dirac operators}
In this subsection, we suppose that $T(D)$ is the one of the following operators:
\begin{align*}
T(D)=(-\Delta)^{s/2},\, T(D)=(-\Delta+1)^{s/2}-1,\, T(D)=\mathcal{D}_0,\, T(D)=\mathcal{D}_{1},
\end{align*}
where $0<s\leq d$.

\begin{proof}[Proof of Theorem \ref{diracth}]
We consider the case when $T(D)=(-\Delta)^{s/2}$ or $T(D)=(1-\Delta)^{s/2}$ only. The case when $T(D)=\mathcal{D}_0$ or $T(D)=\mathcal{D}_1$ is similarly proved if we notice
\begin{align*}
\mathcal{D}_0^2=-\Delta I_{n\times n},\,\,  \mathcal{D}_0^2=(-\Delta+1) I_{n\times n}
\end{align*}
as in the proof of \cite[Theorem 3.1]{C2}. We take a real-valued function $\chi\in C_c^{\infty}(\re^d, [0,1])$ such that $\chi=1$ on $T^{-1}(I)$ and $\supp \chi\subset \re \setminus \L_c(T(D))$. 
Note that $M_{\l}=\{T(\x)=\l\}$ is sphere and hence has non vanishing Gaussian curvature. if $\l\in \s(T(D))\setminus \L_c(T(D))$. Then we apply Theorem \ref{mainprop} with $k=(d-1)/2$ (see \cite[Theorem 1.2.1]{S}) and obtain
\begin{align}\label{eses}
\sup_{z\in I_{\pm}}\|\chi(D)R_0^{\pm}(z)\|_{B(L^{p}(\re^d), L^{q}(\re^d))}<\infty
\end{align}
for $(p,q)\in S_{\frac{d-1}{2}}$. On the other hand, by the support property of $\chi$ and the Hardy-Littlewood-Sobolev inequality, we have
\begin{align}\label{esaway}
\sup_{z\in I_{\pm}}\|(1- \chi(D))R_0(z)\|_{B(L^p(\re^d),L^q(\re^d))}<\infty
\end{align}
if $1/p-1/q\leq s/d$. In fact, if $2\a=-d/2+d/p$ and $2\b=-d/q+d/2$, then
\begin{align*}
\|(1- \chi(D))&R_0(z) \|_{B(L^p(\re^d), L^q(\re^d))}\\
\leq& \|(I-\Delta)^{-\a}\|_{B(L^p(\re^d),L^2(\re^d))} \|(1- \chi(D))(I-\Delta)^{\a+\b}R_0(z)\|_{B(L^2(\re^d))} \\
&\times \|(I-\Delta)^{-\b}\|_{B(L^2(\re^d),L^q(\re^d))}.
\end{align*}
Thus $(\ref{esaway})$ follows from the the Hardy-Littlewood-Sobolev inequality.
Combining $(\ref{eses})$ with $(\ref{esaway})$, we obtain $(i)$. $(ii)$ is similarly proved. 

\begin{lem}\label{dirBir}

\item[$(i)$]
Suppose $2d/(d+1)\leq s< d$. Let $0<\d\leq 1$, $r\in (2d/s, 2(d+1)-4\d]$ and $r_1,r_2\in (1, 2(d+1)]$ satisfying
\begin{align*}
\frac{2}{d+1}\leq \frac{1}{r_1}+\frac{1}{r_2} \leq \frac{s}{d}.
\end{align*}
Then
\begin{align*}
&\sup_{z\in I_{\pm}}\|W_1R_0^{\pm}(z)W_2\|_{B(L^2(\re^d))}\leq C\|W_1\|_{L^{r_1}(\re^d)}\|W_2\|_{L^{r_2}(\re^d)}\\
&\|W_3(R_0^{\pm}(z)-R_0^{\pm}(w))W_4\|_{B(L^2(\re^d))}\leq C|z-w|^{\b_{\d}}\|W_3\|_{L^{r}(\re^d)}\|W_4\|_{L^{r}(\re^d)}
\end{align*}
for $z,w\in I_{\pm}$ with $|z-w|\leq 1$ and $W_1\in L^{r_1}(\re^d)$, $W_2\in L^{r_2}(\re^d)$, $W_3, W_4\in L^{r}(\re^d)$.

Moreover, if $W_1\in L^{r_1}(\re^d)$ and $W_2\in L^{r_2}(\re^d)$, then $W_1R_0^{\pm}(z)W_2\in B_{\infty}(L^2(\re^d))$ follows for $z\in I_{\pm}$ and a map $z\in I_{\pm}\mapsto W_1R_0^{\pm}(z)W_2$ is continuous. 

\item[$(ii)$]
Suppose $0<s<2d/(d+1)$. Let $0<\d\leq 1$, $r\in (1, 2(d+1)-4\d]$ , $r_1,r_2,\in (1,2(d+1)]$ and $r_1', r_2', r'\in [2d/s ,\infty)$ satisfying 
\begin{align*}
\frac{2}{d+1}\leq \frac{1}{r_1}+\frac{1}{r_2},\quad \frac{1}{r_1'}+\frac{1}{r_2'}\leq \frac{s}{d}.
\end{align*}
The all results in Lemma \ref{dirBir} part $(i)$ hold if we replace $L^{r_1}(\re^d)$, $L^{r_2}(\re^d)$ and $L^{r}(\re^d)$ by $L^{r_1}(\re^d)\cap L^{r_1'}(\re^d)$, $L^{r_2}(\re^d)\cap L^{r_2'}(\re^d)$ and $L^{r}(\re^d)\cap L^{r'}(\re^d)$ respectively.
\end{lem} 
\begin{proof}
Note that for $W_1,W_2\in C_c^{\infty}(\re^d)$, it follows that $W_1(1-\chi(D))R_0^{\pm}(z)W_2$ is compact and smooth in $z\in I_{\pm}$ by using $dR_0(z)/dz=R_0(z)^2$ and the Rellich-Kondrachov theorem.
The other parts of the proof are same as in the proof of Corollary \ref{abBir}. 
\end{proof}

Part $(iii)$: Existence and completeness of the wave operators are similarly proved as in the proof of Theorem \ref{discth} $(iv)$ in subsection \ref{discapp} by using Lemma \ref{dirBir}.

Proof of Part $(iv)$ is proved in subsection \ref{Casub}.

\end{proof}

\subsection{Carleman estimate, Proof of Theorem \ref{diracth} $(iv)$}\label{Casub}

First, we give the Carleman estimate for $T(D)$. We recall $\m_{N,\c}(x)=(1+|x|^2)^N(1+\c|x|^2)^{-N}$ and $\L_l=(I-\Delta)^{l/2}$. For $1<p<\infty$ and $l\in \re$, we introduce the standard Sobolev spaces
\begin{align*}
W^{l,p}=\{u\in \mathcal{S}'(\re^d) \mid \L_{l}u\in L^p(\re^d)\},\,\, \|u\|_{W^{l,p}}=\|\L_lu\|_{L^p(\re^d)}.
\end{align*}
We set $p_d=2(d+1)/(d+3)$, $p_d^*=2(d+1)/(d-1)$, $l_d=s/2-d/(d+1)$,
\begin{align*}
X_s=\begin{cases}
W^{-l_d, p_d}+\L_{s/2}\mathcal{B},\,\, \text{if}\,\, 2d/(d+1)\leq s< d,\\
(L^{p_d}(\re^d)\cap L^{2d/(d+s)}(\re^d))+\L_{s/2}\mathcal{B},\,\, \text{if}\,\, 0< s< 2d/(d+1),
\end{cases}
\end{align*}
and
\begin{align*}
X_s^*=\begin{cases}
W^{l_d, p_d^*}\cap \L_{-s/2}\mathcal{B}^*,\,\, \text{if}\,\, 2d/(d+1)\leq s< d,\\
(L^{p_d^*}(\re^d)+ L^{2d/(d-s)}(\re^d))\cap\L_{-s/2}\mathcal{B}^*,\,\, \text{if}\,\, 0< s< 2d/(d+1).
\end{cases}
\end{align*}
By the Sobolev embedding theorem, we have
\begin{align}\label{Sincusion}
X_s\hookrightarrow W^{-s/2,2},\,\, W^{s/2,2}\hookrightarrow X_s^*.
\end{align}

\begin{prop}\label{Car}
Let $N\geq 0$ be a real number satisfying
\begin{align}\label{Ncond}
N<s/2 ,\,\, \text{if}\,\, T(D)=(-\Delta)^{s/2}\,\, \text{with}\,\, s\notin 2\mathbb{N}. 
\end{align}
Then there exists $C_{N,d}>0$ independent of $0<\c\leq 1$ such that
\begin{align*}
\|\m_{N,\c}(x)u\|_{X_s^*}\leq C_{N,d}\|\m_{N,\c}(x) (T(D)-\l) u\|_{X_s}
\end{align*}
for $u\in \mathcal{B}_0^*$ and $|\l|\in I$.
\end{prop}

\begin{rem}
The condition $(\ref{Ncond})$ is needed due to the singularity of the symbol $T(\x)=|\x|^s$ at $\x=0$.
\end{rem}

\begin{proof}
First, we assume $u\in \mathcal{S}(\re^d)$.
Let $\chi_0, \chi_1, \chi_2\in C^{\infty}(\re^d)$ be smooth functions such that $\chi_0,\chi_1\in C_c^{\infty}(\re^d)$ and
\begin{align*}
\chi_0+\chi_1+\chi_2=1,\,\, \chi_0(\x)=1\,\, \text{near}\,\, \x=0,\,\, \chi_1(\x)=1\,\, \text{on}\,\, \supp T^{-1}(I).
\end{align*}
By Lemma \ref{patlem}, it suffices to prove
\begin{align}\label{Carpr}
\|\m_{N,\c}(x)\g(D)u\|_{X_s^*}\leq C_{N,d}\|\m_{N,\c}(x)\g(D) (T(D)-\l) u\|_{X_s}
\end{align}
for $\g\in \{\chi_0, \chi_1, \chi_2\}$. The case when $\g=\chi_1$ directly follows from Proposition \ref{superprop} and Corollary \ref{mucor}. The case when $\g=\chi_2$ follows from Corollary \ref{mucor} and $(\ref{Sincusion})$:
\begin{align*}
\|\m_{N,\c}(x)\chi_2(D)u\|_{X_s^*}\leq& C\|\m_{N,\c}(x)\chi_2(D)u\|_{W^{s/2,2}}\\
=&C\|\L_{s/2}\m_{N,\c}(x)u\|_{L^{2}(\re^d)},\\
\L_{s/2}\m_{N,\c}=&(\L_{s/2}\m_{N,\c}  \L_{-s/2}\m_{N,\c}^{-1})\\
&\times(\m_{N,\c}\L_{s/2}\chi_3(D)(T(D)-\l)^{-1}\m_{N,\c}^{-1}\L_{s/2})\\
&\times\L_{-s/2} \m_{N,\c} \chi_2(D)(T(D)-\l),
\end{align*}
where $\chi_3\in C^{\infty}(\re^d)$ satisfies $\chi_3=1$ on $\supp \chi_2$ and $\supp\chi_3\cap T^{-1}(I)=\emptyset$. Moreover, the $L^2$-boundedness of $\L_{s/2}\m_{N,\c}  \L_{-s/2}\m_{N,\c}^{-1}$ follows from Corollary \ref{mucor} and $L^2$-boundedness of $\m_{N,\c}\L_{s/2}\chi_3(D)(T(D)-\l)^{-1}\m_{N,\c}^{-1}\L_{s/2}$ is proved by mimicking the proof of Corollary \ref{mucor}.

Finally, we deal with the case of $\g=\chi_0$. $(\ref{Carpr})$ with $T(D)\neq (-\Delta)^{s/2}$ or $T(D)=(-\Delta)^{s/2}$ for $s\in 2\mathbb{N}$ is similarly proved as in the proof of $(\ref{Carpr})$ with $\g=\chi_2$. Thus we may assume $T(D)=(-\Delta)^{s/2}$ with $s\notin 2\mathbb{N}$. For its proof, we need some lemmas.

\begin{lem}\label{kersin}
Let $s>0$ and $m\in C^{\infty}(\re^d\setminus \{0\})\cap C_c(\re^d)$ satisfying
\begin{align*}
|\pa_{\x}^{\a}m(\x)|\leq C_{\a}|\x|^{M_{\a}},\,\, M_{\a}=\begin{cases}
0,\,\, \text{if}\,\, \a=0,\\
s-N,\,\, \text{if}\,\, |\a|\geq 1.
\end{cases}
\end{align*}
Then $m(D)(x)=\int_{\re^d}e^{2\pi ix\cdot \x}m(\x)d\x$ satisfies
\begin{align*}
|m(D)(x)|\leq C(1+|x|)^{-s-d}.
\end{align*}
\end{lem}
\begin{proof}
Since $m$ is compactly supported, we may assume $|x|\geq 1$. Let $\chi \in C_c^{\infty}(\re)$ satisfying $\chi(t)=1$ on $|t|\leq 1$ and $\chi(t)=0$ on $|t|\geq 2$. Set $\bar{\chi}=1-\chi$. For $\d>0$, by integrating by parts, we have
\begin{align*}
m(D)(x)=&\frac{x}{|x|^2}\cdot \int_{\re^d}e^{2\pi ix\cdot \x}(-D_{\x}m(\x))d\x\\
=&\frac{x}{|x|^2}\cdot \int_{\re^d}e^{2\pi ix\cdot \x}(\chi(|\x|/\d)+\bar{\chi}(|\x|/\d))(-D_{\x}m(\x))d\x\\
=:&m_1(x)+m_2(x).
\end{align*}
We simply obtain 
\begin{align*}
|m_1(x)|\leq C|x|^{-1}\int_{|\x|\leq 2\d}|\x|^{s-1}d\x\leq C|x|^{-1}\d^{d+s-1}.
\end{align*}
For $M\geq s+d+2$, by integrating by parts, we have
\begin{align*}
|m_2(x)|\leq& C|x|^{-M-1}\sum_{|\a|\leq M}\int_{\re^d}|D_{\x}^{\a} (\bar{\chi}(|\x|/\d)D_{\x}m(\x))|d\x\\
\leq&C|x|^{-M-1}\d^{d+s-1-M}.
\end{align*}
We set $\d=|x|^{-1}$ and conclude $|m(D)(x)|\leq C|x|^{-d-s}$.
\end{proof}

\begin{lem}
Let $m$ be as in Lemma \ref{kersin} and $1<p<\infty$. Moreover, let $0\leq N<s/2$. Then we have
\begin{align*}
&\|\m(x)m(D)\m(x)^{-1}\|_{B(L^{p}(\re^d))}\leq C_{N,m,p}
\end{align*}
for $\m\in \{\m_{N,\c}, \m_{N,\c}^{-1}\}$, where $C_{N,m,p}$ and $C_{N,m}$ are independent of $0<\c\leq 1$ and depends only on $d$, $N$ and $C$ in Lemma \ref{kersin}.
\end{lem}
\begin{proof}
We note that the integral kernel of $\m(x)m(D)\m(x)^{-1}$ is $\m(x)m(D)(x-y)\m(y)^{-1}$ and satisfies
\begin{align*}
|\m(x)m(D)(x-y)\m(y)^{-1}|\leq C(1+|x-y|)^{2N-d-s}
\end{align*}
with $C>0$ independent of $\c>0$. Here we use Lemma \ref{mulemma} $(ii)$ and Lemma \ref{kersin}. We note $2N-s<0$ by the condition $(\ref{Ncond})$. Thus we have $(1+|x|)^{2N-d-s}\in L^1(\re^d)$. By the Young inequality, we obtain the desired result.
\end{proof}
\begin{rem}
Replacing the Young inequality by the O'neil theorem (the Young inequality in the Lorentz spaces), we can relax the condition $(\ref{Ncond})$ as $2N\leq s$.
\end{rem}

We return to the proof of $(\ref{Carpr})$ with $\g=\chi_0$. We take $\chi\in C^{\infty}(\re^d)$ such that $\chi=1$ on $\supp \chi_0$. We learn
\begin{align*}
\L_{s/2}\m_{N,\c}=&(\L_{s/2}\m_{N,\c}  \L_{-s/2}\m_{N,\c}^{-1})\times(\m_{N,\c}\L_{s/2}\chi(D)(T(D)-\l)^{-1}\L_{s/2}\m_{N,\c}^{-1})\\
&\times(\m_{N,\c}\L_{-s/2}\m_{N,\c}^{-1}\L_{s/2})
\times\L_{-s/2} \m_{N,\c} \chi_0(D)(T(D)-\l).
\end{align*}
We set $m(D)=\m_{N,\c}\L_{s/2}\chi(D)(T(D)-\l)^{-1}\L_{s/2}\m_{N,\c}^{-1}$, then $m$ satisfies the assumption of Lemma \ref{kersin}. Thus the inclusions $(\ref{Sincusion})$, Corollary \ref{mucor} and Lemma \ref{kersin} imply $(\ref{Carpr})$ with $\g=\chi_0$. This complete the proof of Proposition \ref{Car} with $u\in \mathcal{S}(\re^n)$.

In order to remove the condition $u\in \mathcal{S}(\re^n)$, we may use the Friedrichs modifier and a cut-off function as in \cite[Proof of Theorem 1.2]{IS}. We omit the detail.
\end{proof}

The next lemma implies that the potential is "admissible".

\begin{lem}\label{adlem}
Suppose $V\in L^{p}(\re^d)$ with $d/s\leq p\leq (d+1)/2$ for $2d/(d+1)\leq s< d$ and $V\in L^{(d+1)/2}(\re^d)\cap L^{d/s}(\re^d)$ for $0<s<2d/(d+1)$. Then we have $V\in B(X_s^{*}, X_s)$. Moreover, for each $\e>0$ and $N\geq 0$ there exists $A_{N,\e}, R_{N,\e}\geq 1$ such that for $\c\in (0,1]$, we have
\begin{align}\label{adlemin}
\|\m_{N,\c}Vu\|_{X_s}\leq \e \|\m_{N,\c}u\|_{X_s^*}+A_{N,\e}\|u\|_{L^2(|x|\leq R_{N,\e})}.
\end{align}
\end{lem}
\begin{proof}
First, we prove
\begin{align}\label{adin}
\|Vu\|_{X_s}\leq \|V\|_{Y_s} \|u\|_{X_s^*},
\end{align}
where $Y_s\in \{L^p(\re^d)\}_{d/s\leq p\leq (d+1)/2}$ for $2d/(d+1)\leq s< d$ and $Y_s=L^{(d+1)/2}(\re^d)\cap L^{d/s}(\re^d)$.
By the Sobolev embedding theorem, we have
\begin{align*}
W^{l_d,p_d^*}\hookrightarrow L^{q^*}(\re^d),\,\, L^{q}(\re^d)\hookrightarrow W^{-l_d,p_d}
\end{align*}
for $2d/(d+s)\leq q\leq p_d$. For $2d/(d+1)\leq s< d$ and $d/s\leq p\leq (d+1)/2$, we set $q_p=2p/(p+1)$. We note $2d/(d+s)\leq q_p\leq p_d$.
By the H\"older inequality, we have
\begin{align*}
\|Vu\|_{L^{q_p}(\re^d)}\leq \|V\|_{L^p(\re^d)} \|u\|_{L^{q_p^*}(\re^d)}.
\end{align*}
We use $X_s^*\hookrightarrow W^{l_d,p_d^*}$ and $W^{-l_d,p_d}\hookrightarrow X_s$ and conclude $V\in B(X_s^*, X_s)$ and $(\ref{adin})$ for $2d/(d+1)\leq s< d$. In order to prove $(\ref{adin})$ with $0<s<2d/(d+1)$, it suffices to prove
\begin{align*}
\|Vu\|_{L^q(\re^d)}\leq \|V\|_{Y_s}\|u\|_{L^r(\re^d)},
\end{align*}
where $q\in\{p_d, 2d/(d+s)\}$ and $r\in \{p_d^*, 2d/(d-s)\}$. This inequality follows from the fact $V\in Y_s= L^{(d+1)/2}(\re^d)\cap L^{d/s}(\re^d)$ and the complex interpolation.

Take $\chi\in C_c^{\infty}(\re^d)$ such that $\chi=1$ on $|x|\leq 1/2$ and $\chi=0$ on $|x|\geq 1$. For $R\geq 1$, we set $V_R=V\chi(x/R)$. Then we use the inclusion $\mathcal{B}\hookrightarrow X_s$ and have
\begin{align*}
\|\m_{N,\c}Vu\|_{X_s}\leq \|V-V_R\|_{Y_s}\|\m_{N,\c}u\|_{X_s^*}+\|\m_{N,\c}V_{R}u\|_{X_s}\\
\leq \|V-V_R\|_{Y_s}\|\m_{N,\c}u\|_{X_s^*}+\|\m_{N,\c}V_{R}u\|_{\mathcal{B}}.
\end{align*}
For each $\e>0$, we take $R>0$ large enough such $\|V-V_R\|_{Y_s}<\e$ and we obtain $(\ref{adlemin})$.

\end{proof}

\begin{proof}[Proof of Theorem \ref{diracth} $(iv)$]
We recall $H=T(D)+V$.
Suppose that $\s_{pp}(H)\setminus \{0\}$ is not discrete in $\re\setminus \{0\}$. Then there exist an orthonormal system $\{u_j\}_{j=1}^{\infty}\subset L^2(\re^d)$, $\d\geq 1$ and $\{\l_j\}_{j=1}^{\infty}\subset \{\l\in \re\mid \d\leq |\l|\leq \d^{-1}\}$ such that $Hu_j=\l_ju_j$. We note $u_j\in L^2(\re^d)\subset \mathcal{B}^*_0$. Let $N\geq 0$ satisfying $(\ref{Ncond})$. Applying Proposition \ref{Car} with $u_j$ and Lemma \ref{adlem} with small $\e>0$, we have
\begin{align*}
\|\m_{N,\c}u_j\|_{X_s^*}\leq C_{N,\e}\|u_j\|_{L^2(\re^d)}
\end{align*}
with $C_{N,\e}$ independent of $\c\in(0,1]$. The inclusion $\L_{-s/2}(1+|x|)^{1/2+\e_1} L^2(\re^d)\hookrightarrow X_s^*$ for $\e_1>0$ implies
\begin{align*}
\|(1+|x|)^{-1/2-\e_1} \L_{s/2}\m_{N,\c}u_j\|_{L^2(\re^d)}\leq C_{N,\e}\|u_j\|_{L^2(\re^d)}.
\end{align*}
Taking $\c\to 0$, we have
\begin{align}\label{discin1}
\|(1+|x|)^{-1/2-\e_1} \L_{s/2}(1+|x|^2)^Nu_j\|_{L^2(\re^d)}\leq C_{N,\e}\|u_j\|_{L^2(\re^d)}=C_{N,\e}.
\end{align}
We take $\e_1$ small enough and $N\geq 0$ satisfying $(\ref{Ncond})$ and $2N>1/2+\e_1$ when $T(D)=(-\Delta)^{s/2}$ with $2s\notin \mathbb{N}$. Then $(\ref{discin1})$ implies that $u_j$ is bounded in $(1+|x|)^{1/2+\e_1-2N}\L_{-s/2} L^2(\re^d)$. Since the inclusion $(1+|x|)^{1/2+\e_1-2N}\L_{-s/2} L^2(\re^d)\hookrightarrow L^2(\re^d)$ is compact, there exists a subsequence $\{u_{j_k}\}_k$ such that $u_{j_k}\to u$ in $L^2(\re^d)$ for some $u\in L^2(\re^d)$. On the other hand, since $u_j$ converges to $0$ in the weak topology of $L^2(\re^d)$, then we have $u=0$. This contradicts to $\|u_j\|_{L^2(\re^d)}=1$.

The same argument implies that the each eigenspace associated with eigenvalue $\l\in \re\setminus \{0\}$ is finite dimensional.

\end{proof}

\subsection{Discrete Schr\"odinger operator}\label{discapp}
In this subsection, we consider the case of $X=\ze$ and consider the discrete Schr\"odinger operators.  

\begin{proof}[Proof of Theorem \ref{discth}]
Part $(ii)$ directly follows from the following lemma.
\begin{lem}
Let $d\geq 4$ and a signature $\pm$. Then maps $z\in \mathbb{C}_{\pm}\setminus\re\mapsto R_0^{\pm}(z)$ are H\"older continuous in $B(L^{p}(\ze^d),L^{p^*}(\ze^d))$ for $1\leq p<3_*$, where $3_*=2d/(d+3)$.
\end{lem}

\begin{proof}
We follow the argument in \cite[Lemma 4.7]{RoS}. We prove the lemma in the case of $+$ only. The case of $-$ is similarly proved. For $1\leq p<3_{*}$, there exists $0<\d\leq 1$ such that $1\leq p<3_{*,\d}$, where
\begin{align*}
3_{*,\d}=\frac{2}{3\d/d+(3+d)/d}.
\end{align*}
We use the following dispersive estimate (\cite{SK}):
\begin{align}\label{disp}
\|e^{it\Delta_d}\|_{B(L^p(\ze^d), L^{p^*}(\ze^d))}\leq C_p\jap{t}^{-\frac{d}{3}(\frac{2}{p}-1)},\,\, 1\leq p\leq 2.
\end{align}
Moreover,
\begin{align}\label{Tay}
|e^{itz}-e^{itz'}|\leq 2^{1-\d}|t|^{\d}|z-z'|^{\d}
\end{align}
holds for $t\geq 0$ and $z,z'\in \mathbb{C}_{+}$ since $|e^{itz}-e^{itz'}|\leq 2$ and $|e^{itz}-e^{itz'}|\leq |t||z-z'|$.
By $(\ref{disp})$ and $(\ref{Tay})$, we have
\begin{align*}
&\|R_0^{+}(z)-R_0^{+}(z')\|_{B(L^{p}(\re^d), L^{p^*}(\re^d))}\\
&=\left\|\int_{0}^{\infty}(e^{itz}-e^{itz'})e^{it\Delta_d}dt\right\|_{B(L^{p}(\re^d), L^{p^*}(\re^d))}\\
&\leq C_p2^{1-\d}|z-z'|^{\d}\int_0^{\infty}|t|^{\d}\jap{t}^{-\frac{d}{3}(\frac{2}{p}-1)}dt<\infty
\end{align*}
for $1\leq p<3_{*,\d}$. This completes the proof.
\end{proof}

Now we prove part $(i)$. The above lemma implies that 
\begin{align}\label{recont}
\lim_{\e\to 0, \e>0}\|R_{0}^{\pm}(\l\pm i\e)-R_0^{\pm}(\l\pm i0)\|_{B(L^p(\ze^d), L^{p^*}(\ze^d))}=0,\,\,\l\in \re,\,\, 1\leq p<3_*,
\end{align}
where we recall $R_0^{\pm}(\l\pm i0)$ are Fourier multipliers of the distributions $(h_0(\x)-(\l\pm i0))^{-1}$.
We also use the uniform bounds (\cite[Proposition 3.3]{TT}):
\begin{align}\label{endbound}
\sup_{z\in \mathbb{C}_{\pm}\setminus \re}\|R_0^{\pm}(z)\|_{B(L^{3_*}(\ze^d), L^{3^*}(\ze^d))}<\infty,
\end{align}
where $3^*=2d/(d-3)$. By $(\ref{recont})$ and $(\ref{endbound})$, taking a limiting argument, we have
\begin{align*}
\sup_{z\in \mathbb{C}_{\pm}}\|R_0^{\pm}(z)\|_{B(L^{3_*}(\ze^d), L^{3^*}(\ze^d))}<\infty.
\end{align*}
This proves part $(i)$.

Note that part $(iii)$ with $V\in L^{p}(\ze^d)$ for $1\leq p<d/3$ follows from part $(ii)$ and the H\"older inequality. Part $(iii)$ with $V\in L^{d/3}(\ze^d)$ follows from the following lemma.

\begin{lem}\label{discear}
Let $d\geq 4$ and a signature $\pm$. For $W_1, W_2\in L^{2d/3}(\ze^d)$, a map $z\in \mathbb{C}_{\pm}\mapsto W_1R_0^{\pm}(z)W_2\in B_{\infty}(L^2(\ze^d))$ is continuous.
\end{lem}
\begin{proof}
Take sequences of finitely supported potentials $W_{1,n}, W_{2,n}$ such that $W_{j,n}\to W_j$ in $L^{2d/3}(\ze^d)$ as $n\to \infty$ for $j=1,2$. For $z,z'\in \mathbb{C}_{\pm}$, the H\"older inequality implies
\begin{align*}
\|&W_1(R_0^{\pm}(z)-R_0^{\pm}(z'))W_2\|_{B(L^2(\ze^d))}\\
\leq& 2\|W_1-W_{1,n}\|_{L^{2d/3}(\ze^d)}\|W_2\|_{L^{2d/3}(\ze^d)}\sup_{z\in \mathbb{C}_{\pm}}\|R_0^{\pm}(z)\|_{B(L^{3_*}(\ze^d), L^{3^*}(\ze^d))}\\
+&2\|W_2-W_{2,n}\|_{L^{2d/3}(\ze^d)}\sup_{n}(\|W_{1,n}\|_{L^{2d/3}(\ze^d)})\sup_{z\in \mathbb{C}_{\pm}}\|R_0^{\pm}(z)\|_{B(L^{3_*}(\ze^d), L^{3^*}(\ze^d))}\\
+&\|W_{1,n}(R_0^{\pm}(z)-R_0^{\pm}(z'))W_{2,n}\|_{B(L^2(\ze^d))}\\
=:& I_1+I_2+I_3.
\end{align*}
Now we let $\e>0$. We fix a large $n$ such that $I_1+I_2$ is smaller than $2\e/3$. Since $W_{1,n}$ and $W_{2,n}$ are finitely supported, the previous lemma implies that $W_{1,n}(R_0^{\pm}(z)-R_0^{\pm}(z'))W_{2,n}$ is H\"older continuous in $B(L^2(\ze^d))$. Thus there exists $\d>0$ such that $|z-z'|<\d$ implies 
\begin{align*}
I_3=\|W_{1,n}(R_0^{\pm}(z)-R_0^{\pm}(z'))W_{2,n}\|_{B(L^2(\ze^d))}<\e/3.
\end{align*}
Thus we conclude that maps $z\in \mathbb{C}_{\pm}\mapsto W_1R_0^{\pm}(z)W_2$ are continuous.
\end{proof}

It remains to prove $(iv)$. 
We follow the argument as in \cite{KY} and \cite{KM}. Let $V\in L^{d/3}(\ze^d)$ be a real-valued function. Set $W_1=(\sgn V)|V|^{1/2}\in L^{2d/3}(\ze^d)$, $W_2=|V|^{1/2}\in L^{2d/3}(\ze^d)$, $H=H_0+V$ and $R(z)=(H-z)^{-1}$ for $z\in \mathbb{C}\setminus \re$. We note that for $\pm \Im z> 0$
\begin{align}\label{reseq}
W_1R_0^{\pm}(z)W_2-W_1R(z)W_2=W_1R(z)W_2W_1R_0^{\pm}(z)W_2.
\end{align}
By part $(iii)$,  it follows that $W_1R_0^{\pm}(z)W_2$ is continuous in $z\in I_{\pm}$ and hence is a compact operator . In addition, $I+W_1R_0^{\pm}(z)W_2$ is invertible in $B(L^2(\ze^d))$ for $z\in \mathbb{C}\setminus\re$ due to the Birman-Schwinger principle. In fact, if $I+W_1R_0^{\pm}(z)W_2$ is not invertible at $z\in \mathbb{C}\setminus\re$, then the compactness of $W_1R_0^{\pm}(z)W_2$ implies that $I+W_1R_0^{\pm}(z)W_2$ has a non-trivial kernel. Then it follows that $R(z)$ has a non-trivial kernel by the Birman-Schwinger principle. However, this contradicts to the self-adjointness of $H_0+V$. Moreover, if we set 
\begin{align*}
\s_{\mathrm{BS}}(H)=\s_{\mathrm{BS}}^{\pm}(H)=\{\l \in \re\mid \Ker_{L^2(\ze^d)}(I+W_1R_0^{\pm}(z)W_2)\neq 0\},
\end{align*}
we see that $\s_{\mathrm{BS}}(H)$ is a closed set with Lebesgue measure zero by Proposition \ref{cpxprop}. Since $W_1R_0^{\pm}(z)W_2\in B_{\infty}(L^2(\ze^d))$ for $z\in I_{\pm}$, $I+W_1R_0^{\pm}(z)W_2$ is a Fredholm operator with index $0$. Thus $(\ref{reseq})$ gives
\begin{align*}
W_2R(z)W_2= W_2R_0^{\pm}(z)W_2(I+W_1R_0^{\pm}(z)W_2)^{-1},\,\, z\in I_{\pm}\setminus \s_{\mathrm{BS}}(H_0).
\end{align*}
Let $[a,b]\subset I\setminus \s_{\mathrm{BS}}(H_0)$ with $a<b$. Since $(I+W_1R_0^{\pm}(z)W_2)^{-1}$ is continuous in $z\in [a,b]_{\pm}$, then 
\begin{align*}
\sup_{z\in [a,b]_{\pm}}\|(I+W_1R_0^{\pm}(z)W_2)^{-1}\|_{B(L^2(\ze^d))}<\infty.
\end{align*}
Combining this with the part $(i)$ and H\"older's inequality, we obtain
\begin{align*}
\sup_{z\in [a,b]_{\pm}}\|W_2R(z)W_2\|_{B(L^2(\ze^d))}<\infty.
\end{align*}
Since $|W_1|=|W_2|$, then 
\begin{align*}
\sup_{z\in [a,b]_{\pm}}\|W_{i_1}R(z)W_{i_2}\|_{B(L^2(\ze^d))}<\infty.
\end{align*}
for $i_1,i_2=1,2$. By \cite[Theorem XIII. 30, 31]{RS}, the local wave operators $s-\lim_{t\to\pm \infty}e^{itH}e^{-itH_0}E_{H_0}((a,b))$ exist and are complete, where $E_{H_0}(J)$ is the spectral projection to the interval $J\subset \re$ associated with $H_0$. Since $[0,4d]\setminus \L_c(H_0)\cup \s_{\mathrm{BS}}(H)$ is a countable union of such interval $(a,b)$, the wave operators $W_{\pm}=s-\lim_{t\to\pm \infty}e^{itH}e^{-itH_0}$ exist and are complete.

\end{proof}

As an application of Theorem \ref{mainprop}, we prove the further estimates of the uniform resolvent estimates for the discrete Schr\"odinger operators.

\begin{prop}\label{disclow}
Suppose $I\subset (0,4)\cap (4(d-1), 4d)$ if $d=2$ and $I\subset (0,2)\cap (4d-2, 4d)$ if $d\geq 3$. If $\supp\chi\subset h_0^{-1}(I)$, then 
\begin{align*}
\sup_{z\in I_{\pm}}\|\chi(D)R_0^{\pm}(z)\|_{B(L^p(\ze^d), L^q(\ze^d))}<\infty.
\end{align*}
holds for $(1/p,1,q)\in S_{(d-1)/2}$.
\end{prop}
\begin{proof}
Let $\l\in I$. As is proved in \cite[Lemma 4.3]{IM}, all principal curvatures of $M_{\l}=\{h=\l\}$ are non-vanishing. By Example \ref{ex}, we obtain the desired result.
\end{proof}

\appendix

\section{Some estimates for $\c_{z,\pm}$}

In this section, we give proofs of the estimates for $\c_{z,\pm}$ which is needed for the proof of Theorem \ref{mainprop}.

If necessary we take $\supp \chi$ small, we may assume $X=\re$. We recall the situation of the proof of Theorem \ref{mainprop}.
Set 
\begin{align*}
\tilde{\chi}(\x',\x_d,\l)=\frac{\chi^2(\x',\x_d+h_{\l}(\x')) }{e(\x', \x_d+h_{\l}(\x'))},\,\, b(\x',\x_d,\l)=e(\x', \x_d+h_{\l}(\x')))^{-1}.
\end{align*}
Note that $b$ is real-valued and $\min_{(\x',\x_d)\in\supp\chi(\cdot,\cdot,\l) \tilde{\chi} , \l\in I}b(\x',\x_d,\l)>0$.
Recall that 
\begin{align*}
\c_{z,\pm}(\x',x_d)=\int_{\re}\frac{e^{2\pi ix_d\x_d}\tilde{\chi}(\x',\x_d,\l) }{\x_d-i(\Im z)b(\x',\x_d,\l)}d\x_d,\,\, \Re z=\l, \,\, \pm \Im z\geq 0.
\end{align*}
Here if $\pm \Im z=0$, we interpret $\c_{z,\pm}$ as
\begin{align*}
\c_{z,\pm}(\x',x_d)=&\int_{\re}\frac{e^{2\pi ix_d\x_d}\chi^2(\x',\x_d+h_{\l}(\x')) }{e(\x', \x_d+h_{\l}(\x'))\x_d\mp i0}d\x_d\\
=&\int_{\re}\frac{e^{2\pi ix_d\x_d}\tilde{\chi}(\x',\x_d,\l)) }{\x_d\mp i0}d\x_d,
\end{align*}
where $(\x_d\mp i0)^{-1}$ denote the distributions $\lim_{\e>0,\e\to 0}(\x_d\mp i\e)^{-1}$. In order to estimate $\c_{z,\pm}$, we need some lemmas.

\begin{lem}\label{Hil}
Let $\g, \g_1\in C_c^{\infty}(\re)$ and $\m_1, \m_2\in\re\setminus\{0\}$. Then
\begin{align*}
&|\int_{\re}\g(\m_1 y_d)\mathrm{p.v.}\frac{e^{2\pi iy_d\x_d}}{y_d}dy_d|\leq \pi\|\hat{\g}\|_{L^1(\re)},\,\, |\int_{\re}\mathrm{p.v.}\frac{e^{2\pi iy_d\x_d}}{y_d}dy_d|=\pi\\
&|\int_{\re}\g(\m_1y_d)\g_1(\m_2y_d)\mathrm{p.v.}\frac{e^{2\pi iy_d\x_d}}{y_d}dy_d|\leq \pi\|\hat{\g}\|_{L^1(\re)}\|\hat{\g_1}\|_{L^1(\re)},\\
\end{align*}
\end{lem}

\begin{proof}
We leran
\begin{align*}
|\int_{\re}\mathrm{p.v.}\frac{1}{y_d} \g(y_d)e^{2\pi iy_d\x_d}dy_d|=&\pi |\int_{\re}\sgn(\x_d-\y_d)\hat{\g}(-\y_d)d\y_d|\\
\leq&\pi\|\hat{\g}\|_{L^1(\re)}.
\end{align*}
By scaling, we obtain the first inequality.
The second equality follows from $\mathcal{F}(\mathrm{p.v.}\frac{1}{y_d})(\x_d)=-i\pi\sgn(\x_d)$. The third inequality follows from the first inequality and the Young inequality:
\begin{align*}
\|\hat{\g}\hat{\g_1}\|_{L^1(\re)}=&\|\hat{\g}\ast \hat{\g_1}\|_{L^1(\re)}\\
\leq&\|\hat{\g}\|_{L^1(\re)}\|\hat{\g_1}\|_{L^1(\re)}.
\end{align*}

\end{proof}

\begin{lem}\label{leminte}
Let $\m\in\re\setminus \{0\}$ and $\f, a, a_1\in C_c^{\infty}(\re)$ such that $a, a_1$ are real-valued and $a, a_1>0$ on $\supp \f$.
\begin{itemize}
\item[$(i)$] There exists $C>0$ independent of $x_d\in\re$, $\f, a$ and $\m\neq 0$ such that
\begin{align}\label{Aintes}
|\int_{\re}\frac{e^{2\pi ix_d\x_d} \f(\m\x_d)}{\x_d-ia(\m \x_d)}d\x_d|\leq C(\sup_{\x_d\in \re}|\frac{\f(\x_d)}{a(\x_d)}|+\|\hat{\f}\|_{L^1(\re)}+\sup_{\x_d\in \re}|\f(\x_d)a(\x_d)|).
\end{align}
\item[$(ii)$] Let $l\geq 2$ be an integer. Then there exists $C'>0$ independent of $x_d\in\re$, $\f, a$, $l$ and $\m\neq 0$ such that
\begin{align}\label{Aintes2}
|\int_{\re}\frac{e^{2\pi ix_d\x_d} \f(\m\x_d)}{(\x_d-ia(\m \x_d))^l}d\x_d|\leq C'(\sup_{\x_d\in \re}|\frac{|\f(\x_d)|}{|a(\x_d)|^l}+\|\f\|_{L^{\infty}(\re)}).
\end{align}
\item[$(iii)$] Let $l_1, l_2\geq 1$ be an integer. Then there exists $C''>0$ independent of $x_d\in\re$, $\f, a$, $l$ and $\m\neq 0$ such that
\begin{align}\label{Aintes3}
|\int_{\re}\frac{e^{2\pi ix_d\x_d} \f(\m\x_d)}{(\x_d-ia(\m \x_d))^{l_1}(\x_d-ia_1(\m \x_d))^{l_2}}d\x_d|\leq C''(\sup_{\x_d\in \re}|\frac{|\f(\x_d)|}{|a(\x_d)|^{l}}+\|\f\|_{L^{\infty}(\re)}).
\end{align}
\end{itemize}
\end{lem}
\begin{proof}
$(i)$ Take $\g\in C_c^{\infty}(\re, [0,1])$ such that $\g=1$ on $|t|\leq 1$ and $\g=0$ on $|t|\geq 2$. Since $a$ is real-valued, then
\begin{align*}
|\int_{\re}\frac{e^{2\pi ix_d\x_d} \f(\m\x_d)\g(\x_d)}{\x_d-ia(\m \x_d)}d\x_d|\leq&\int_{\re}\frac{|\f(\m\x_d)\g(\x_d)|}{|a(\m\x_d)|}d\x_d\\
\leq&\sup_{\x_d\in \re}|\frac{\f(\x_d)}{a(\x_d)}|\|\g\|_{L^1(\re)}.
\end{align*}
We note that
\begin{align*}
\int_{\re}\frac{e^{2\pi ix_d\x_d} \f(\m\x_d)(1-\g(\x_d))}{\x_d-ia(\m \x_d)}d\x_d=&\int_{\re}\frac{e^{2\pi ix_d\x_d} \f(\m\x_d)(1-\g(\x_d))}{\x_d}d\x_d\\
&+i\int_{\re}\frac{e^{2\pi ix_d\x_d} \f(\m\x_d)a(\m\x_d)(1-\g(\x_d))}{\x_d(\x_d-ia(\m \x_d))}d\x_d\\
=&:I_1+I_2.
\end{align*}
By Lemma \ref{Hil}, we have
\begin{align*}
|I_1|=& |\int_{\re}\mathrm{p.v.}\frac{e^{2\pi ix_d\x_d} \f(\m\x_d)}{\x_d}d\x_d-\int_{\re}\mathrm{p.v.}\frac{e^{2\pi ix_d\x_d} \f(\m\x_d)\g(\x_d)}{\x_d}d\x_d|\\
\leq&\pi\|\hat{\f}\|_{L^1(\re)}(1+\|\hat{\g}\|_{L^1(\re)}).
\end{align*}
Moreover, since $a$ is real-valued, we have
\begin{align*}
|I_2|\leq \sup_{\x_d\in \re}|\f(\x_d)a(\x_d)|\int_{\re}\frac{1-\g(\x_d)}{\x_d^2}d\x_d.
\end{align*}
Thus we set
\begin{align*}
C= \max(\|\g\|_{L^1(\re)}, \pi(1+\|\hat{\g}\|_{L^1(\re)}), \int_{\re}\frac{1-\g(\x_d)}{\x_d^2}d\x_d),
\end{align*}
and obtain $(\ref{Aintes})$.

$(ii)$ follows from $(iii)$.

$(iii)$ Let $\g$ be as above. Then
\begin{align*}
|\int_{\re}\frac{e^{2\pi ix_d\x_d} \f(\m\x_d)\g(\x_d)}{(\x_d-ia(\m \x_d))^{l_1}(\x_d-ia_1(\m \x_d))^{l_2}}d\x_d|\leq \sup_{\x_d\in \re}|\frac{|\f(\x_d)|}{|a(\x_d)|^{l_1}|a_1(\x_d)|^{l_2}}\|\g\|_{L^1(\re)}.
\end{align*}
Moreover, since $a,a_1$ is real-valued and $l_1+l_2\geq 2$, then
\begin{align*}
|\int_{\re}\frac{e^{2\pi ix_d\x_d} \f(\m\x_d)(1-\g(\x_d))}{(\x_d-ia(\m \x_d))^{l_1}(\x_d-ia_1(\m \x_d))^{l_2}}d\x_d|\leq& \|\f\|_{L^{\infty}(\re)}\int_{\re}\frac{1-\g(\x_d)}{|\x_d|^{l_1+l_2}}d\x_d\\
\leq&\|\f\|_{L^{\infty}(\re)}\int_{\re}\frac{1-\g(\x_d)}{|\x_d|^2}d\x_d.
\end{align*}
Thus we set $C''=\max(\|\g\|_{L^1(\re)}, \int_{\re}\frac{1-\g(\x_d)}{|\x_d|^2}d\x_d)$ and obtain $(\ref{Aintes3})$. 

\end{proof}

The main result of this section is the following proposition.

\begin{prop}\label{appth}
Fix a signature $\pm$.
\begin{itemize}
\item[$(i)$] For $\a\in \mathbb{N}^{d-1}$, there exists $C_{\a}>0$ such that
\begin{align}\label{games1}
|\pa_{\x'}^{\a}\c_{z,\pm}(\x',x_d)|\leq C_{\a}
\end{align}
for $z\in I_{\pm}$, $x_d\in \re$ and $\x'\in \re^{d-1}$.
\item[$(ii)$] For $\a\in \mathbb{N}^{d-1}$, there exists $C_{\a}'>0$ such that
\begin{align}\label{games2}
|\pa_{\x'}^{\a}(\c_{z,\pm}(\x',x_d)-\c_{w,\pm}(\x',x_d))|\leq C_{\a}'(1+|x_d|)|z-w|
\end{align}
for $z, w\in I_{\pm}$ with $|z-w|\leq 1$, $x_d\in \re$ and $\x'\in \re^{d-1}$.
\end{itemize}

\end{prop}
\begin{rem}
Let $0\leq \d\leq 1$. Combining $(\ref{games1})$ with $(\ref{games2})$, we have
\begin{align}\label{games3}
|\pa_{\x'}^{\a}(\c_{z,\pm}(\x',x_d)-\c_{w,\pm}(\x',x_d))|\leq C_{\a}^{1-s}(C_{\a}')^s(1+|x_d|)^{\d}|z-w|^{\d}.
\end{align}

\end{rem}

\begin{proof}
$(i)$ 
We follow the argument of the proof of \cite[(3.10)]{C2}. We may assume $0\leq \pm \Im z\leq 1$. First, we consider the case of $\pm \Im z=0$. In this case, the claim follows from the fact that
\begin{align*}
\|\int_{\re}\frac{e^{2\pi ix_d\x_d}}{\x_d\mp i0}d\x_d\|_{L^{\infty}(\re_{x_d})}<\infty
\end{align*}
and that $\tilde{\chi}$ is smooth with respect to $(\x, \x_d,\l)\in \re^d\times I$ and has a compact support with respect to $(\x',\x_d)$-variable which is bounded in $\l \in I$.

We take $\g\in C_c^{\infty}(\re, [0,1])$ such that $\g(\x_d)=1$ on $|\x_d|\leq 1$. We learn 
\begin{align*}
\c_{z,\pm}(\x',x_d)=\int_{\re}\frac{e^{2\pi i (\Im z) x_d\x_d}\tilde{\chi}(\x',(\Im z) \x_d,\l) }{\x_d-ib(\x',(\Im z) \x_d,\l)}d\x_d.
\end{align*}
We note that $\pa_{\x'}^{\a}\c(\x',x_d)$ is a linear combination of the form
\begin{align*}
\int_{\re}\frac{e^{2\pi i (\Im z) x_d\x_d}(\pa_{\x'}^{\a_0}\tilde{\chi})(\x',(\Im z) \x_d,\l)\prod_{j=1}^l(\pa_{\x'}^{\a_j}b)(\x',(\Im z)\x_d,\l) }{(\x_d-ib(\x',(\Im z) \x_d,\l))^l}d\x_d,
\end{align*}
where $l\geq 1$ is an integer and $\a_j\in \mathbb{N}^{d-1}$ for $j=0,...,l$. Applying Lemma \ref{leminte} $(i)$ if $l=1$ and $(ii)$ if $l>1$ with $\f(\x_d)=(\pa_{\x'}^{\a_0}\tilde{\chi})(\x', \x_d,\l)\prod_{j=1}^l(\pa_{\x'}^{\a_j}b)(\x',\x_d,\l)$, $a(\x_d)=b(\x',\x_d, \l)$ and $\m=\Im z$, we obtain $(\ref{games1})$ with $|\a|\geq 1$.

$(ii)$ We set $\l=\Re z$ and $\s=\Re w$.
 We take $0<\e$ such that 
\begin{align*}
\min_{(\x',\x_d)\in \supp \chi(\cdot, \cdot,\l), |z-w|\leq \d}|b(\x',\x_d,\s)|>0.
\end{align*}
Then we may assume $|z-w|<\e$.
In fact, in order to prove $(ii)$, we use $(i)$ if $|z-w|\geq \e$. Note that
\begin{align*}
\c_{z,\pm}(\x',x_d)-\c_{w,\pm}(\x',x_d)=&J_1(x_d)+J_2(x_d)+J_3(x_d),
\end{align*}
where we set
\begin{align*}
J_1(x_d)=&\int_{\re}e^{2\pi ix_d\x_d}(\frac{\tilde{\chi}(\x',\x_d,\l)}{\x_d-i(\Im z)b(\x',\x_d,\l)}-\frac{\tilde{\chi}(\x',\x_d,\l)}{\x_d-i(\Im w)b(\x',\x_d,\l) })d\x_d\\
J_2(x_d)=&\int_{\re}e^{2\pi ix_d\x_d}\frac{\tilde{\chi}(\x',\x_d,\l)-\tilde{\chi}(\x',\x_d,\s)}{\x_d-i(\Im w)b(\x',\x_d,\l)}d\x_d\\
=&\int_{\re}e^{2\pi i(\Im w) x_d\x_d}\frac{\tilde{\chi}(\x',(\Im w)\x_d,\l)-\tilde{\chi}(\x',(\Im w)\x_d,\s)}{\x_d-ib(\x',(\Im w)\x_d,\l)}d\x_d\\
J_3(x_d)=&\int_{\re}e^{2\pi ix_d\x_d}\tilde{\chi}(\x',\x_d,\s)(\frac{1}{\x_d-i(\Im w)b(\x',\x_d,\l)}-\frac{1}{\x_d-i(\Im w)b(\x',\x_d,\s)})d\x_d\\
=&\int_{\re}e^{2\pi i(\Im w)x_d\x_d}\frac{i\tilde{\chi}(\x',(\Im w)\x_d,\s)(b(\x',(\Im w)\x_d, \l)-b(\x',(\Im w)\x_d, \s)) }{(\x_d-ib(\x',(\Im w)\x_d,\l))(\x_d-ib(\x',(\Im w)\x_d,\s))}d\x_d.
\end{align*}
First, we estimate $J_2$. Similarly to the proof of $(i)$, $\pa_{\x'}^{\a}J_2(\x')$ is a finite sum of the form
\begin{align*}
\int_{\re}&\frac{e^{2\pi i (\Im w) x_d\x_d}((\pa_{\x'}^{\a_0}\tilde{\chi})(\x',(\Im w) \x_d,\l)-(\pa_{\x'}^{\a_0}\tilde{\chi})(\x',(\Im w) \x_d,\s)) }{(\x_d-ib(\x',(\Im w) \x_d,\l))^l}\\
&\times \prod_{j=1}^l(\pa_{\x'}^{\a_j}b)(\x',(\Im w)\x_d,\l)d\x_d,
\end{align*}
where $l\geq 1$ is an integer and $\a_j\in \mathbb{N}^{d-1}$ for $j=0,...,l$. We apply Lemma \ref{leminte} $(i)$ if $l=1$ and $(ii)$ $l\geq 2$ and obtain
\begin{align}\label{J_2}
|\pa_{\x'}^{\a}J_2(\x')|\leq C_{\a}'|z-w|
\end{align}
with $C_{\a}'>0$ independent of $x_d\in \re$, $\x_d\in \re^{d-1}$ and $z,w\in I_{\pm}$ with $|z-w|\leq \d$.

Next, we estimate $J_3$. $\pa_{\x'}^{\a}J_3$ is a linear combination of the form
\begin{align*}
\int_{\re}&\frac{e^{2\pi i (\Im w) x_d\x_d}(\pa_{\x'}^{\a_0}\tilde{\chi})(\x',(\Im w) \x_d,\s)\pa_{\x'}^{\a_2}(b(\x',(\Im w)\x_d, \l)-b(\x',(\Im w)\x_d, \s))  }{(\x_d-ib(\x',(\Im w) \x_d,\l))^{l_1}(\x_d-ib(\x',(\Im w)\x_d, \s))^{l_2}}\\
&\times\prod_{j=2}^{l_1+l_2+1}(\pa_{\x'}^{\a_j}b)(\x',(\Im w)\x_d,\l)d\x_d,
\end{align*}
where $l_1,l_2\geq 1$ are integers and $\a_j\in \mathbb{N}^{d-1}$ for $j=0,...,l_1+l_2+1$.
We apply Lemma \ref{leminte} $(iii)$ and obtain
\begin{align}\label{J_3}
|\pa_{\x'}^{\a}J_3(\x')|\leq C_{\a}'|z-w|
\end{align}
with $C_{\a}'>0$ independent of $x_d\in \re$, $\x_d\in \re^{d-1}$ and $z,w\in I_{\pm}$ with $|z-w|\leq \e$.

Finally, we estimate $J_1$.
Note that $|\pa_{\x'}^{\a}J_1(x_d)|\leq 2C_0$ by $(i)$. Thus it suffices to prove that $|\pa_{\x'}^{\a}J_1'(x_d)|\leq C_{\a}'|\Im z-\Im w|$. We learn
\begin{align*}
\frac{J_1'(x_d)}{2\pi i}=&\int_{\re}e^{2\pi ix_d\x_d}(\frac{\x_d\tilde{\chi}(\x',\x_d,\l)}{\x_d-i(\Im z)b(\x',\x_d,\l)}-\frac{\x_d\tilde{\chi}(\x',\x_d,\l)}{\x_d-i(\Im w)b(\x',\x_d,\l) })d\x_d\\
=&\int_{\re}e^{2\pi ix_d\x_d}\frac{i(\Im z-\Im w)\x_d\tilde{\chi}(\x',\x_d,\l)b(\x',\x_d,\l)}{(\x_d-i(\Im z)b(\x',\x_d,\l))(\x_d-i(\Im w)b(\x',\x_d,\l))}d\x_d\\
=&\int_{\re}e^{2\pi i(\Im w)x_d\x_d}\frac{i(\Im z-\Im w)\x_d\tilde{\chi}(\x',(\Im w)\x_d,\l)b(\x',(\Im w)\x_d,\l)}{(\x_d-i\frac{\Im z}{\Im w}b(\x',( \Im w)\x_d,\l))(\x_d-ib(\x',(\Im w)\x_d,\l))}d\x_d.
\end{align*}
Thus $\pa_{\x'}^{\a}J_1'(x_1)/(-2\pi |\Im z-\Im w|)$ is a linear combination of the form
\begin{align*}
(\frac{\Im z}{\Im w})^{l_1}\int_{\re}&e^{2\pi i(\Im w)x_d\x_d}\frac{\x_d\pa_{x'}^{\a_0}\tilde{\chi}(\x',(\Im w)\x_d,\l)\pa_{\x'}^{\a_2}b(\x',(\Im w)\x_d,\l)}{(\x_d-i\frac{\Im z}{\Im w}b(\x',( \Im w)\x_d,\l))^{l_1}(\x_d-ib(\x',(\Im w)\x_d,\l))^{l_2}}\\
&\times \prod_{j=2}^{l_1+l_2+1}\pa_{\x'}^{\a_j}b(\x',\x_d,\l) d\x_d,
\end{align*}
where $l_1,l_2\geq 1$ are integers, $\a_j\in \mathbb{N}^{d-1}$ for $j=1,...,l_1+l_2+1$. Applying Lemma \ref{leminte} $(i)$ and $(ii)$ with 
\begin{align*}
\f(\x_d)=(\Im z)^{l_1}\frac{\x_d\pa_{x'}^{\a_0}\tilde{\chi}(\x',(\Im w)\x_d,\l)\pa_{\x'}^{\a_2}b(\x',(\Im w)\x_d,\l)}{(\x_d-i(\Im z) b(\x',\x_d,\l))^{l_1}},
\end{align*}
$a(\x_d)=b(\x',\x_d,\l)$, $l=l_2$ and $\m=\Im w$, we have $|\pa_{\x'}^{\a}J_1'(x_d)|\leq C_{\a}'|\Im z-\Im w|$. This completes the proof.

\end{proof}

\section{Complex analysis}

We define $\log^+t=\log t$ if $1\leq t$, $\log^+t=0$ if $0<t\leq 1$ and $\log^{-}t=\log t-\log^+t$. 

\begin{lem}\label{Jen}
Let $f:\{z\in \mathbb{C}\mid |z|\leq 1\}\to \mathbb{C}$ be a continuous function which is holomorphic on $\{|z|<1\}$ and has no zero on $\{|z|<1\}$. Then $f(e^{i\theta})\neq 0$ for almost everywhere $\theta\in [-\pi,\pi)$.

\end{lem}
\begin{proof}
We follow the argument of \cite[Theorem 17.17]{Rud}.
By the mean value properties of the harmonic function, we have
\begin{align}\label{log1}
\log|f(0)|=&\frac{1}{2\pi}\int_{-\pi}^{\pi}\log|f(re^{i\theta})|d\theta\\
=&\frac{1}{2\pi}\int_{-\pi}^{\pi}\log^{+}|f(re^{i\theta})|d\theta-\frac{1}{2\pi}\int_{-\pi}^{\pi}\log^{-}|f(re^{i\theta})|d\theta \nonumber
\end{align}
for $0<r<1$. On the other hand, by using $x\leq e^x$ for $x\in \re$ and Jensen's inequality, we have
\begin{align*}
\frac{1}{2\pi}\int_{-\pi}^{\pi}\log^{+}|f(re^{i\theta})|d\theta\leq& \exp(\frac{1}{2\pi}\int_{-\pi}^{\pi}\log^{+}|f(re^{i\theta})|d\theta)\\
\leq&\frac{1}{2\pi}\int_{-\pi}^{\pi}|f(re^{i\theta})|d\theta.
\end{align*}
By Fatou's lemma and $(\ref{log1})$, we obtain $\log|f(e^{i\theta})|\in L^1([-\pi,\pi))$. In particular, $\log|f(e^{i\theta})|<\infty$ for almost everywhere $\theta\in [-\pi,\pi)$. Thus $f(e^{i\theta})\neq 0$ for almost everywhere $\theta\in [-\pi,\pi)$.

\end{proof}

\begin{cor}\label{cpxcor}
Let $J=(a,b)$ be an open interval and $r=(b-a)/2$. 
Let $f:\{z\in \mathbb{C}\mid |z-(a+b)/2|\leq r,\, \pm \Im z\geq 0\}\to \mathbb{C}$ be a continuous function which is holomorphic and has no zero on $\{|z-(a+b)/2|<r, \Im z>0\}$. Then $f(\l)\neq 0$ for almost everywhere $\l\in J$.
\end{cor}
\begin{proof}
For simplicity, we assume $a=-1$ and $b=1$. 
Define $\k_1:D=\{|z|<1, \Im z>0\}\to \{\Im z>0\}$ and $\k_2:\{\Im z>0\}\to \{|z|<1\}$ by $\k_1(z)=(1+z)^2/(1-z)^2$ and $\k_2(z)=(z-i)/(z+i)$. Then $\k=\k_2\circ\k_1$ is biholomorphic from $\{|z|<1, \Im z>0\}$ to $\{|z|<1\}$ and  homeomorphic from $\{|z|\leq 1, \Im z\geq 0\}$ to $\{|z|\leq 1\}$. Moreover, since 
\begin{align*}
\k^{-1}(w)=\frac{\sqrt{i\frac{1+w}{1-w}}-1}{\sqrt{i\frac{1+w}{1-w}}+1}
\end{align*}
where we take a branch such that $\Im \sqrt{z}>0$, then $\k^{-1}|_{|z|=1}:\{|z|=1\}\to \bar{D}\setminus D$ is H\"older continuous. Thus $\k^{-1}|_{|z|=1}$ maps sets of Lebesgue measure zero to sets of Lebesgue measure zero. By Lemma \ref{Jen}, we obtain the desired result.
\end{proof}

Next proposition is a variant of \cite[Lemma 4.20]{KK}. See also \cite[Proposition 4.6]{KM}.

\begin{prop}\label{cpxprop}
Let $Z$ be a Banach space and fix a sgnature. For $J\subset \re$ be an open set, we denote $J_{\pm}=\{z\in \mathbb{C}\mid \Re z\in J, \pm\Im z\geq 0\}$. Let $K:J_{\pm}\to B_{\infty}(Z)$ be continuous and holomorphic on $\{\pm\Im z>0\}$. If $I+K(z)$ has a inverse in $B(Z)$ for each $z\in \{\pm \Im z>0\}$, then $\Gamma_0=\{\l\in \re\mid I+K(\l)\,\text{is not invertible}\}$ is a closed set with Lebesgue measure zero.
\end{prop}

\begin{proof}
Since the set of all invertible operators in $B(Z)$ is open and since $K$ is continuous, then $\Gamma_0$ is closed. Thus it suffices to prove that the Lebesgue measure of $\Gamma_0$ is zero. 
Note that $I+K(\l)$ is not invertible if and only if $-1$ is in the spectrum of $K(\l)$ for $\l\in \Gamma_0$. Fix $\l\in \Gamma_0$. Since $K(\l)$ is compact, there exists a circle $C_{\l}$ enclosing $-1$ such that $C_{\l}$ is contained in the resolvent set of $K(\l)$. Since $K$ is continuous, there exists $r_{\l}>0$ such that $C_{\l}$ is contained in the resolvent set of $K(z)$ for $z\in \overline{B_{r_\l}^{\pm}(\l)}$ where $B_{r_\l}^{\pm}(\l)=\{z\in \mathbb{C}\mid \pm\Im z\geq 0,\, |z-\l|<r_{\l}\}$. We define
\begin{align*}
P_z=\frac{1}{2\pi i}\int_{C_{\l}}(w-K(z))^{-1}dw,
\end{align*}
then $z\in \overline{B_{r_\l}^{\pm}(\l)}\mapsto P_z\in B(Z)$ is analytic in $B_{r_\l}^{\pm}(\l)\setminus \re$ and continuous in $B_{r_\l}^{\pm}(\l)$. Note that $n_0=\dim \Ran P_{z}<\infty$ is independent of $z\in \overline{B_{r_\l}^{\pm}(\l)}$. Set $Z_z=\Ran P_z$ and fix a linear isomorphism $\Pi_{\l}:\mathbb{C}^{n_0}\to Z_{\l}$. We choose $r_{\l}$ smaller such that $I+P_{\l}(P_z-P_{\l})$ has an inverse in $B(Z_{\l})$. Then $\Theta_{z}=P_{z}|_{Z_{\l}}:Z_{\l}\to Z_z$ is a linear isomorphism with its inverse
\begin{align*}
(I+P_{\l}(P_z-P_{\l}))^{-1}P_{\l}:Z_z\to Z_{\l}.
\end{align*}
Now we set 
\begin{align*}
X(z)=\Pi_{\l}^{-1}\Theta_z^{-1}(I+K(z))\Theta_z\Pi_{\l}
\end{align*}
for $z\in \overline{B_{r_\l}^{\pm}(\l)}$. Then $X$ is continuous on $\overline{B_{r_\l}^{\pm}(\l)}$ and analytic in $B_{r_\l}^{\pm}(\l)$. Moreover, $\det X(z)$ is also continuous on $\overline{B_{r_\l}^{\pm}(\l)}$ and analytic in $B_{r_\l}^{\pm}(\l)$. We note that $\det X(z)=0$ if and only if $-1$ is in the spectrum of $K(z)$. By Corollary \ref{cpxcor} and the compactness argument, we conclude that the Lebesgue measure of $\Gamma_0$ is zero.

\end{proof}

\end{document}